\documentclass[10pt]{article}
\pdfoutput=1
\usepackage[latin1]{inputenc}
\usepackage{amsmath,amsfonts,amssymb,amsthm,graphicx}
\usepackage{verbatim}
\usepackage{mathtools}
\usepackage{tabularx,booktabs}
\usepackage[margin=1in]{geometry}
\usepackage{subcaption}
\usepackage{float}

\newtheorem{theorem}{Theorem}[section]
\newtheorem{lemma}[theorem]{Lemma}
\newtheorem{cor}[theorem]{Corollary}
\newtheorem{prop}[theorem]{Proposition}

\theoremstyle{definition}
\newtheorem{definition}[theorem]{Definition}

\theoremstyle{remark}

\numberwithin{equation}{section}

\DeclareMathOperator{\lcm}{lcm}

\DeclareMathOperator{\probOP}{\mathbb{P}}
\DeclareMathOperator{\expectOP}{\mathbb{E}}
\DeclareMathOperator{\varOP}{Var}
\DeclareMathOperator{\covOP}{Cov}
\DeclareMathOperator{\pdim}{pdim}

\newcommand{\littleo}[1]{o\left( #1 \right)}

\newcommand{\prob}[1]{\probOP\left[ #1\right]}
\newcommand{\expect}[1]{\expectOP\left[ #1\right]}
\newcommand{\var}[1]{\varOP\left[ #1\right]}
\newcommand{\cov}[1]{\covOP\left[ #1\right]}
\newcommand{\kring}{k[x_1,\ldots,x_n]}

\newcommand{\N}{\mathbb{N}}

\newcommand{\Z}{\mathbb{Z}}
\newcommand{\sse}{\subseteq}

\newcommand{\gradedmodel}{\ensuremath\mathcal{M}(n,D,p)}

\title{Average Behavior of Minimal Free Resolutions of Monomial Ideals}
\author{Jes\'us A. De Loera, Serkan Ho\c{s}ten, Robert Krone and Lily Silverstein}
\begin{document}
\maketitle

\begin{abstract}
We show that, under a natural probability distribution, random monomial
ideals will almost always have minimal free resolutions of maximal length; that is,
the projective dimension will almost always be $n$, where $n$ is the number of
variables in the polynomial ring. As a consequence we prove that
Cohen-Macaulayness is a rare property. We characterize when a random 
monomial ideal is generic/strongly generic, and when it is Scarf---i.e., when the
algebraic Scarf complex of $M\subset S=k[x_1,\ldots,x_n]$ gives a minimal free
resolution of $S/M$. It turns out, outside of a very specific ratio
of model parameters, random monomial ideals are Scarf only when they are generic. 
We end with a discussion of the average magnitude of Betti numbers.

\end{abstract}

\section{Introduction}

Minimal free resolutions are an important and central topic in commutative algebra.  For instance, in the setting of modules over finitely generated graded $k$-algebras, the numerical data of these resolutions determine the Hilbert series, Castelnuovo-Mumford regularity and other fundamental invariants. Minimal free resolutions also provide a starting place for a myriad of homology and cohomology computations. 
For the essentials on minimal free resolutions in our setting, see \cite{Eisenbudbook}. 

Much has been written about the extremal behavior of minimal free resolutions  on monomial ideals (e.g., \cite{bigatti,hulett,pardue}), and about their combinatorial and computational properties (e.g., \cite{bayerMONOMIALRES,herzog+hibi, lascala+stillman, millerCCA}). 
In this paper we formalize and explore the \emph{average} behavior of minimal free resolutions with respect to a probability distribution on monomial ideals. Monomial ideals are a natural setting for this exploration; they define modules over polynomial rings that are, in many ways, the simplest possible, and yet they are general enough to capture the full spectrum of values for many algebraic properties \cite{coxlittleoshea,herzog+hibi}. 

In \cite{deloeraRMI}, the authors introduced a probabilistic model  for monomial ideals and characterized the distribution of several invariants including the Hilbert function, the Krull dimension/codimension, and the number of minimal generators. In their  model with parameters $n$, $D$, and $p$, a random monomial ideal in $n$ indeterminants is defined by independently choosing generators of degree at most $D$ with probability $p$ each. Based on extensive simulations, they stated conjectures on several properties related to minimal free resolutions, including projective dimension and Cohen-Macaulayness. This work presents answers to these conjectures, for a special case of the \emph{graded model} described in \cite{deloeraRMI}. We also settle a question about (strong) genericity, and describe the properties of random Scarf complexes.

Throughout this paper, we consider random monomial ideals in $n$ variables which are minimally generated in a single degree $D$, where each monomial of degree $D$ has the same probability $p$ of appearing as a minimal generator. 
That is, a minimal generating set $G$ is sampled according to
$$
\prob{x^\alpha\in G}
=
\begin{cases}
p & |\alpha|=D \\
0 & \text{ otherwise},
\end{cases}
$$ for all $x^\alpha\in S=\kring$. We then set $M=\langle G \rangle$. Given the three parameters $n$, $D$, and $p$, we denote this model by $\gradedmodel$, and write $M\sim\gradedmodel$.
When we consider the asymptotic behavior of $n\to\infty$ or $D\to\infty$, we think of $p$ as a function of $n$ or $D$, respectively, and write $p=p(n)$ or $p=p(D)$. For two functions $f(z)$, $g(z)$ of the same variable, we write $f(z)\ll g(z)$, equivalently $g(z)\gg f(z)$, if $\lim_{z\to\infty}f(z)/g(z)=0$.

The \emph{projective dimension} of $S/I$, $\pdim(S/I)$, is the minimum length of a free resolution of $S/I$. Hilbert's celebrated \emph{syzygy theorem} (see Section 19.2 in \cite{Eisenbudbook}) established that $\pdim(S/I) \leq n$ for any $I\sse S$.
In our first result, we prove the existence of a threshold for the parameter $p=p(D)$, above which almost every random monomial ideal has projective dimension equal to $n$. 
\begin{theorem}\label{thm:proj-dim}
Let $S=k[x_1,\ldots,x_n]$, $M\sim\gradedmodel$, and $p=p(D)$. As $D\to\infty$, $p = D^{-n+1}$ is a threshold for the projective dimension of $S/M$.  If $p \ll D^{-n+1}$ then $\pdim(S/M)=0$ asymptotically almost surely and if $p \gg D^{-n+1}$ then $\pdim(S/M) = n$ asymptotically almost surely. 
\end{theorem}
In other words, the case of equality in Hilbert's syzygy theorem is the most typical situation for non-trivial ideals. 

Prior experiments had indicated that   Cohen-Macaulayness is a rare property among random monomial ideals \cite{deloeraRMI}. 
Using Theorem \ref{thm:proj-dim} we prove this is indeed the case.

\begin{cor}\label{cor:CM}
Let $S=k[x_1,\ldots,x_n]$, $M\sim\gradedmodel$, and $p=p(D)$. If $D^{-n+1}\ll  p \ll 1$, then asymptotically almost surely $S/M$ is not Cohen-Macaulay.
\end{cor}

One of the key combinatorial tools for computing the minimal free resolution of a monomial ideal is the \emph{Scarf complex}, introduced in
\cite{bayerMONOMIALRES}. The Scarf complex is a simplicial complex, with vertices given by the minimal generators of an ideal, that defines a chain complex contained in the minimal free resolution. In general, however, the Scarf complex does not give a resolution of $S/M$. When it does, the Scarf complex is actually a minimal free resolution of $S/M$, and we say that $M$ \emph{is Scarf}.
If a monomial ideal $M$ is \emph{generic} or \emph{strongly generic}, then $M$ is Scarf \cite{bayerMONOMIALRES}. The next two theorems characterize when $M\sim\gradedmodel$ is generic, and when it is Scarf.
 
 \begin{theorem}\label{thm:scarf}
 	Let $S=k[x_1,\ldots,x_n]$, $M\sim\gradedmodel$, and $p=p(D)$. If $p \gg  D^{-n+2-1/n}$ then $M$ is not Scarf asymptotically almost surely. 
 \end{theorem}

\begin{theorem}\label{thm:generic}
Let $S=k[x_1,\ldots,x_n]$, $M\sim\gradedmodel$, and $p=p(D)$. As $D\to\infty$, $p = D^{-n+3/2}$ is a threshold for $M$ being generic and for $M$ being strongly generic.
  If $p \ll D^{-n+3/2}$ then $M$ is generic or strongly generic asymptotically almost surely, and if $p \gg D^{-n+3/2}$ then $M$ is neither generic nor strongly generic asymptotically almost surely.
\end{theorem}

Notice that Theorem \ref{thm:scarf} does not provide a threshold result for being Scarf. Nevertheless, taken together with Theorem \ref{thm:generic} it indicates that being Scarf is almost equivalent to being generic in our probabilistic model. Monomial ideals that are not generic but Scarf live in the small range
$D^{-n+3/2} \ll p \ll D^{-n+2-1/n}$. This narrow ``twilight zone" can be seen in the following figures as the transition region where black, grey, and white are all present. 

\begin{figure}[h] 
\includegraphics[width=\textwidth]{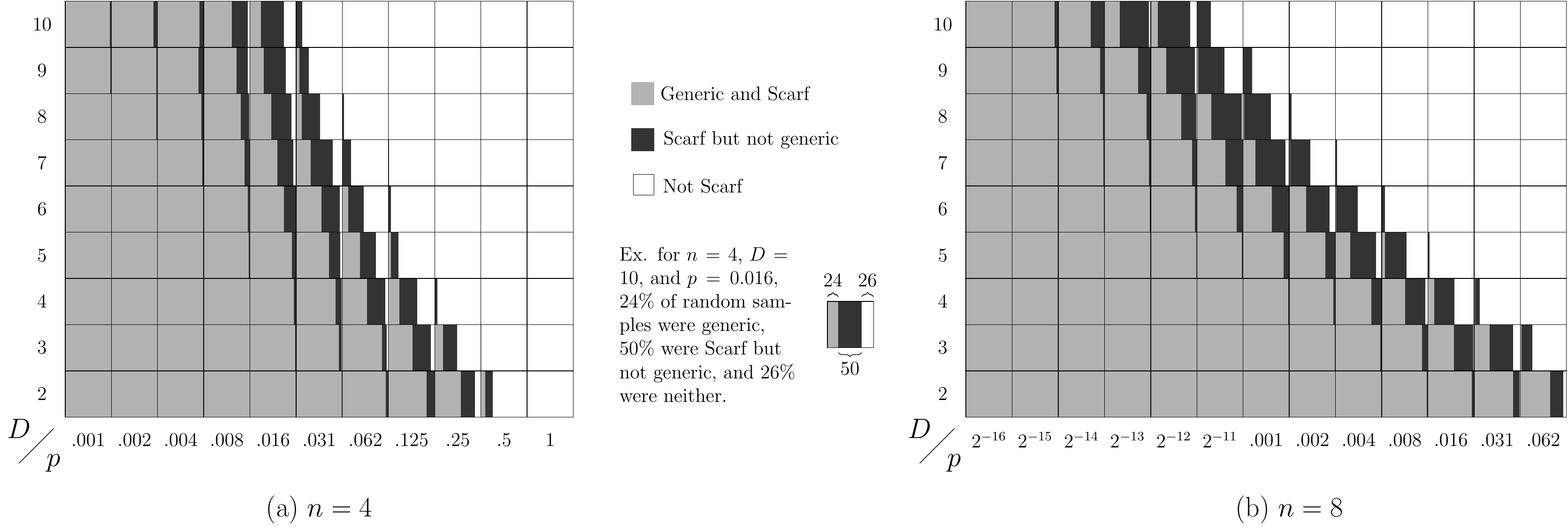}
\caption{Generic versus Scarf monomial ideals in computer simulations of the graded model.}
\label{fig:scarf-and-generic}
\end{figure}

As an application of the probabilistic method, by choosing parameters in the twilight zone, we can generate countless examples of ideals with the unusual property of being Scarf but not generic. An example found while creating Figure \ref{fig:scarf-and-generic} is $I=\langle x_1^4x_3x_5^5$, $ x_1x_2^2x_3^2x_6^4x_8$, $ x_2^3x_5^2x_6^3x_7x_8$, $ x_1^3x_5^2x_7^2x_8^3$, $ x_2x_3x_4^3x_6x_8x_9^3$, $ x_1x_3^4x_4x_6^2x_8x_{10}$, $ x_1x_3x_4^2x_5x_6x_8^3x_{10}$, $ x_2x_3x_6^3x_8^4x_{10}$, $ x_4x_5^5x_7x_{10}^3$, $ x_1x_5^4x_{10}^5\rangle\sse k[x_1,\ldots x_{10}]$, which has the following total Betti numbers:
$$
\begin{array}{c|ccccccccc}
i&0&1&2&3&4&5&6&7&8\\\beta_i&1&10&45&114&168&147&75&20&2,
\end{array}
$$
and is indeed Scarf. Creating---or even verifying---such examples by hand would be a rather difficult task!

\vskip 10pt

We would like to conclude pointing to some earlier work in the probabilistic study of syzygies and minimal resolutions.
To our knowledge, the first investigation of ``average'' homological behavior was that of Ein, Erman and Lazarsfeld \cite{ein+erman+lazarsfeld}, who studied the ranks of syzygy modules of smooth projective varieties. Their conjecture---that these ranks are asymptotically normal as the positivity of the embedding line bundle grows---is supported by their proof of asymptotic normality for the case of random Betti tables. Their random model is based on the elegant Boij-S\"oderberg theory established by Eisenbud and Schreyer \cite{eisenbud+schreyer}; for a fixed number of rows, they sample by choosing Boij-S\"oderberg coefficients independently and uniformly from $[0,1]$, then show that with high probability the Betti table entries become normally distributed as the number of rows goes to infinity. Further support to this conjecture is the paper of Erman and Yang \cite{ermanFLAG}, which uses the probabilistic method to exhibit concrete examples of families of embeddings that demonstrate this asymptotic normality.

\section{The projective dimension of random monomial ideals}

\subsection{Witness sets for ${\bf \pdim(S/M) = n}$}

In what follows let $S=\kring$, and let $M= \langle G \rangle \sse S$ be a monomial ideal with minimal generating set $G$.  We summarize a criterion for $G$, given in 2017 by Alesandroni, that is
equivalent to the statement $\pdim(S/M)=n$. See \cite{alesandroniDOM,alesandroniPDIM} for details and proofs.

First, a few definitions. Let $L$ be a set of monomials. 
An element $m = x_1^{\alpha_1} \cdots x_n^{\alpha_n} \in L$ is a \emph{dominant monomial} (in $L$) if there is a variable $x_i$
such that the $x_i$ exponent of $m$, $\alpha_i$, is strictly larger than the $x_i$ exponent of any other monomial in $L$.
If every $m\in L$ is a dominant monomial, then $L$ is a \emph{dominant set}.
For example, $L_1=\{x_1^3x_2x_3^2,x_2^2x_3,x_1x_3^3\}$ is a dominant set in $k[x_1,x_2,x_3]$, but $L_2=\{x_1^3x_2x_3^2,x_2^2x_3,x_1^3x_3^3\}$ is not.
For monomials $m=x_1^{\alpha_1}\cdots x_n^{\alpha_n}$ and $m'=x_1^{\beta_1}\cdots x_n^{\beta_n}$, we say that $m$ \emph{strongly divides} 
$m'$ if $\alpha_i<\beta_i$ whenever $\alpha_i\neq 0$. Thus, $x_1x_3$ strongly divides $x_1^2x_3^3$, but $x_1x_3$ does not strongly divide $x_1x_3^3$.

We can now state the characterization.
\begin{theorem}\label{thm:pdim-equiv-dom-set-strong} \cite[Theorem 5.2, Corollary 5.3]{alesandroniPDIM}
Let $M \sse S$ be a monomial ideal minimally generated by $G$. Then $\pdim(S/M) = n$ if and only if there is a subset $L$ of $G$ with the following properties:
\begin{enumerate}
\item $L$ is dominant.
\item $|L|=n$.
\item No element of $G$ strongly divides $\lcm(L)$.
\end{enumerate}
More precisely, if $L\sse G$ satisfies conditions 1, 2 and 3, then the minimal free resolution of $S/M$ has a basis element with multidegree $\lcm(L)$ in homological degree $n$. On the other hand, if there is a basis element with multidegree $x^\alpha$ and homological degree $n$, then $G$ must contain some $L'$ satisfying 1, 2, 3 and the condition $\lcm(L')=x^\alpha$. 
\end{theorem}
The latter, stronger characterization is important to our results on Scarf complexes (Section \ref{sec:gen-scarf}). In this section, we care only that $\pdim(S/M)=n$ is equivalent to the existence of a subset of generators satisfying the conditions of Theorem \ref{thm:pdim-equiv-dom-set-strong}. 
Since we frequently discuss such sets, we use the following terminology throughout the paper.


\begin{definition}\label{defn:witness}
When $L$ is any set of minimal generators of $M$ that satisfies the three conditions of Theorem \ref{thm:pdim-equiv-dom-set-strong}, 
then $L$ witnesses $\pdim(S/M) = n$, and we say $L$ is a \emph{witness set}. The monomial $x^\alpha\in S$ is a \emph{witness lcm} if $L$ is a witness set and $x^\alpha=\lcm(L)$.
\end{definition}

The distinction between witness sets and witness lcm's is important, as several witness sets can have a common lcm. 
We found it useful to think of the event ``$x^\alpha$ is a witness lcm" in geometric terms, as illustrated in Figure \ref{fig:witness-geometrically} for the case of $n=3$. 

The monomials of total degree $D$ are represented as lattice points in a regular $(n-1)$-simplex with side lengths $D$. Given $x^\alpha=x_1^{\alpha_1}\cdots x_n^{\alpha_n}$, 
the $n$ inequalities $x_1\leq \alpha_1,\ldots,x_n\leq \alpha_n$ determine a new regular simplex $\Delta_\alpha$ (shaded). If $L$ is a dominant set that satisfies 
$|L|=n$ and $\lcm(L)=x^\alpha$, then $L$ must contain exactly one lattice point from the interior of each facet of $\Delta_\alpha$. 
(Monomials on the boundary of a facet are dominant in more than one variable.) Meanwhile, the strong divisors of $x^\alpha$ are the lattice points in the interior of $\Delta_\alpha$. 
The event ``$x^\alpha$ is a witness lcm" occurs when at least one generator is chosen in the interior of each  facet of $\Delta_\alpha$, and no generators are chosen in the interior of $\Delta_\alpha$.

\begin{figure}[h]
	\begin{subfigure}[t]{0.32\textwidth}
		\includegraphics[width=\textwidth]{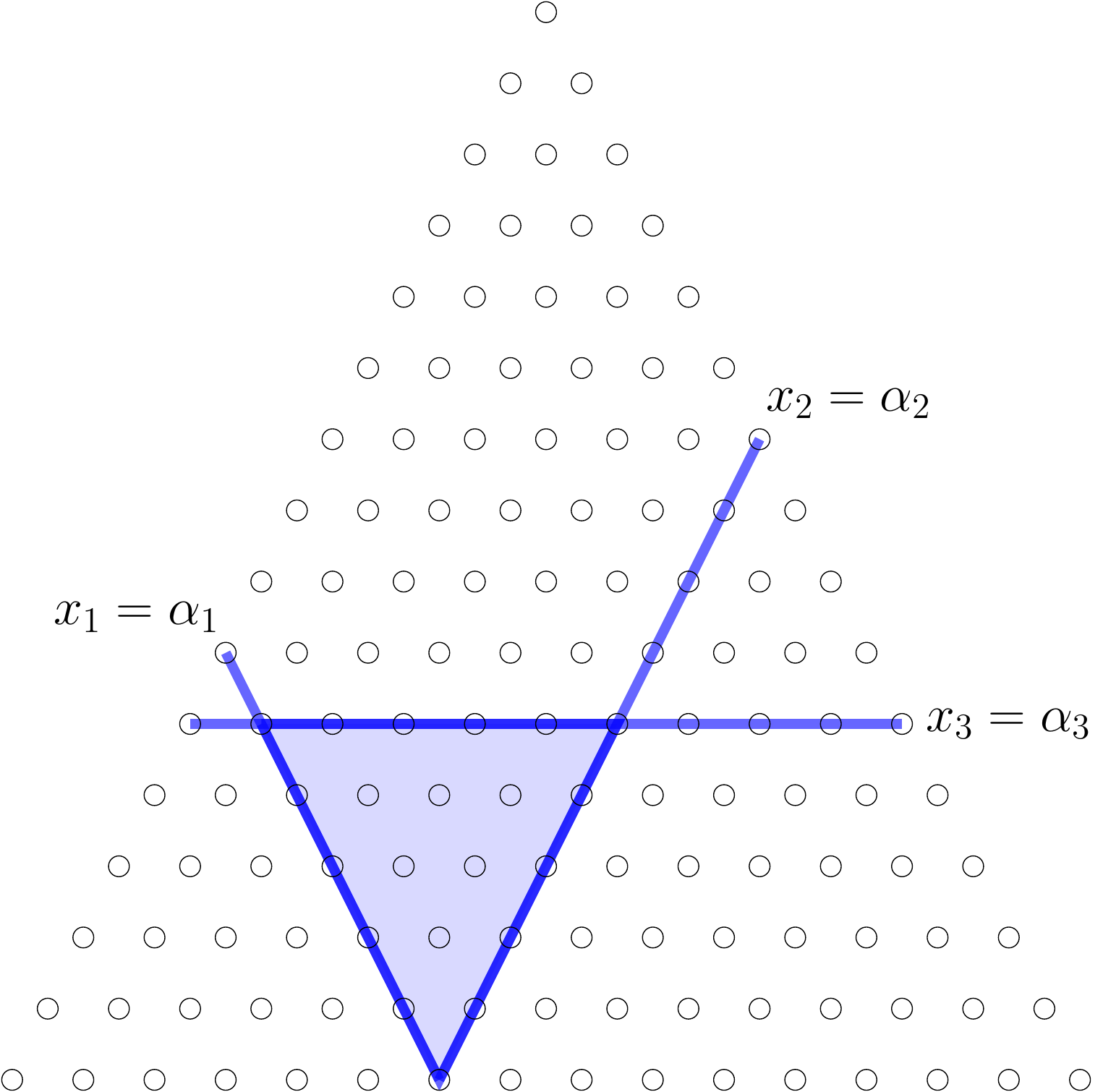}
		\caption{The simplex $\Delta_\alpha$ associated with the witness lcm $x^\alpha=x_1^{\alpha_1}x_2^{\alpha_2}x_3^{\alpha_3}$ is defined by facets $x_i\leq \alpha_i$ for $i=1,2,3$.}
	\end{subfigure}
	\hfill
	\begin{subfigure}[t]{0.32\textwidth}
		\includegraphics[width=\textwidth]{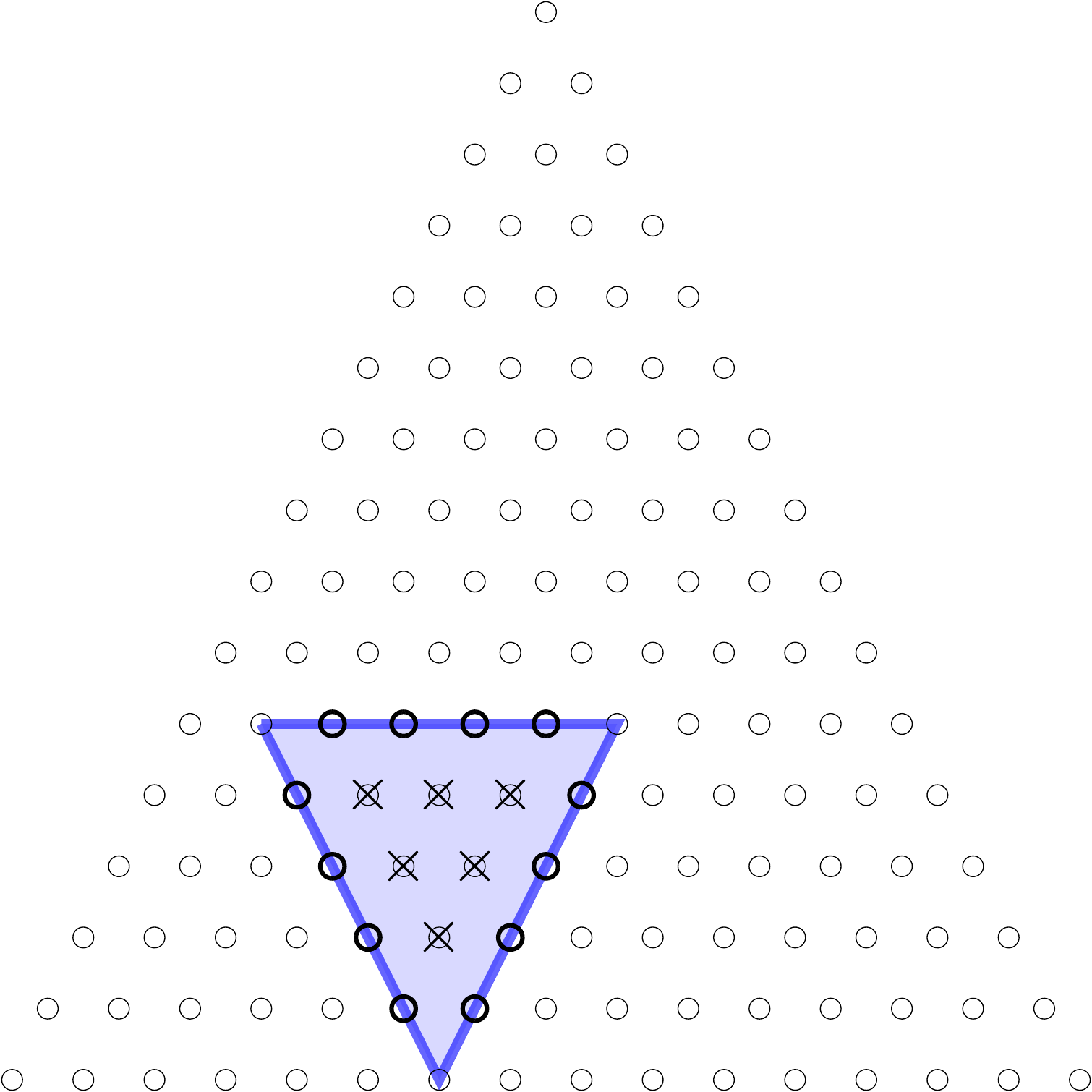}
		\caption{For $x^\alpha$ to be a witness lcm, at least one monomial on the interior of each facet (bold outline) must be chosen, and none of the interior monomials (crossed out) can be chosen.}
	\end{subfigure}
	\hfill
	\begin{subfigure}[t]{0.32\textwidth}
		\includegraphics[width=\textwidth]{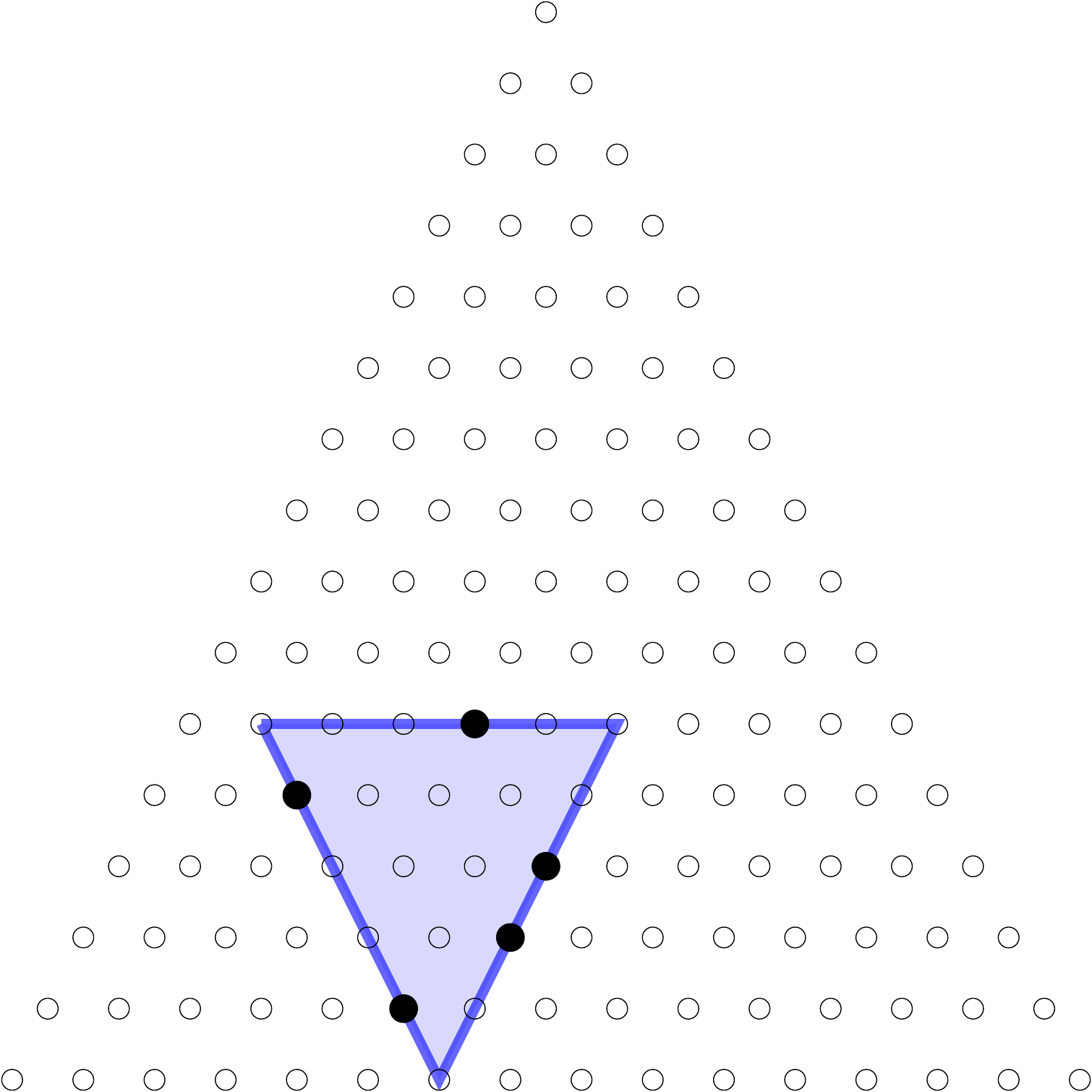}
		\caption{A situation where $x^\alpha$ is a witness lcm. Notice there are four different ways to choose one generator from each facet, so there are four witness sets with this lcm.}
	\end{subfigure}
	\caption{Geometric interpretation of a witness set. }
	\label{fig:witness-geometrically}
\end{figure}

We will make use of some common probability laws, and so we review them briefly here.  The first is {\em Markov's inequality} which states that if $X$ is a nonnegative random variable and $a \geq 0$, then
 \[ a\prob{X \geq a} \leq \expect{X}. \]
The second is the {\em union bound}.  If $X_1,\ldots,X_r$ are a collection of indicator variables, the probability that any of the events occur (the union) is at most the sum of the probabilities that each one occurs.  When the variables are independent and identically distributed (i.i.d.) and each has probability $p$ of occurring, then the union bound implies the following useful inequality:
 \[ 1 - (1-p)^r \leq rp. \]
We will also use the \emph{second moment method}.  This is a special case of Chebyshev's inequality  and asserts that
\begin{equation}\label{eq:second-moment}
\prob{X=0}
\leq
\frac{\var{X}}{\expect{X}^2},
\end{equation}
for a nonnegative, integer-valued random variable $X$. 

\subsection{Most resolutions are as long as possible}

This section comprises the proof of Theorem \ref{thm:proj-dim} and two
of its consequences.  First we show that for $p$ below the announced threshold, usually $\pdim(S/M) = 0$.  
Let 
\[ m_n(D) = \binom{D+n-1}{n-1} \]
denote the number of monomials in $n$ variables of degree $D$.
This is a polynomial in $D$ of degree $n-1$ and can be bounded, for $D$ sufficiently large, by
\begin{equation}\label{eq:mon-count}
 \frac{1}{(n-1)!}D^{n-1} \leq m_n(D) \leq \frac{2}{(n-1)!}D^{n-1}.
\end{equation}

\begin{prop} \label{prop:pdimzero}
 If $p \ll D^{-n+1}$ then $\pdim(S/M) = 0$ asymptotically almost surely as $D \to \infty$.
\end{prop}
\begin{proof}
For each $x^\alpha\in S$, let $X_\alpha$ be the random variable indicating that $x^\alpha\in G$ ($X_\alpha=1$) or $x^\alpha\not\in G$ ($X_\alpha=0$). 
We define $X=\sum_{\alpha\in S} X_\alpha$, so that $X$ records the cardinality of the random minimal generating set $G$. By Markov's inequality,
$$
\prob{X> 0}=\prob{X\geq 1}\leq \expect{X} 
=\sum_{\substack{\alpha\in S\\ |\alpha|=D}} \expect{X_\alpha}
= m_n(D)p.
$$
Letting $D\to\infty$, we have
$$
\lim_{D\to\infty}
\prob{X> 0}
=
\lim_{D\to\infty}
m_n(D)p
=
0,
$$
since $p \ll D^{-n+1}$. So $| G|=0$, equivalently $M= \langle 0 \rangle$, with probability converging to $1$ as $D\to\infty$. Therefore below the threshold $D^{-n+1}$, 
almost all random monomial ideals in our model have $\pdim(S/M)=0$.
\end{proof}

For the case $p \gg D^{-n+1}$, we use the second moment method.
Recall that $x^\alpha\in S$ is a {witness lcm} to $\pdim(S/M)=n$ if and only if there is a dominant set $L\sse G$ with 
$|L|=n$, $\lcm(L)=x^\alpha$, and no generator in $G$ strongly divides $x^\alpha$.
For each $\alpha$, we define an indicator random variable $w_\alpha$ that equals $1$ if $x^\alpha$ is a witness lcm  and $0$ otherwise.
Next we define $W_a$, for integers $a>1$, and $W$ by
$$
W_a
=
\sum_{\substack{|\alpha|=D+a \\ \alpha_i\geq a\, \forall i}} w_\alpha,
\hspace*{6em}
W=
\sum_{a=n-1}^A
W_a
$$
where $A = \lfloor (p/2)^{-\frac{1}{n-1}}\rfloor - n$.
The random variable $W_a$ counts \emph{most} witness lcm's of degree $D+a$. The reason for the restriction $\alpha_i\geq a$ is easily explained geometrically. 
In general, the probability that $x^\alpha$ is a witness lcm depends only on the side length of the simplex $\Delta_\alpha$  (see Figure \ref{fig:witness-geometrically}). 
If, however, the facet defining inequalities of $\Delta_\alpha$ intersect outside of the simplex of monomials with degree $D$, the situation is more complicated and has many 
different cases. The definition of $W_a$ bypasses these cases, and this does not change the asymptotic analysis.

In Lemma \ref{lemma:wa-bound}, we compute the order of $\prob{w_\alpha}$ and use this to prove  that $\expect{W}\to\infty$ as $D\to\infty$ in Lemma \ref{lemma:nbound-W}. 
Then in Lemma \ref{lemma:variance-W}, we prove $\var{W}=\littleo{\expect{W}^2}$ and thus that the right-hand side of  (\ref{eq:second-moment}) goes to $0$ as $D\to\infty$.
In other words, $\prob{W> 0}\to 1$, meaning that $M\sim\gradedmodel$ will have at least one witness to $\pdim(S/M)=n$ with probability converging to $1$ as $D\to\infty$. 
This proves the second side of the threshold and establishes the theorem.

We first give the value of $\prob{w_\alpha}$ for an exponent vector $\alpha$ with $|\alpha| = D+a$ and $\alpha_i \geq a$ for all $i$.  
The monomials of degree $D$ that divide $x^\alpha$ form the simplex $\Delta_\alpha$, and those that strongly divide $x^\alpha$ form the interior of $\Delta_\alpha$. 
Thus there are $m_n(a)$ divisors and $m_n(a-n)$ strong divisors of  $x^\alpha$ in degree $D$.  Recall that for $x^\alpha$ to be a witness lcm, for each variable $x_i$ there must 
be at least one monomial $x^\beta$ in $G$ with $x^\beta$ in the relative interior of the facet of $\Delta_\alpha$ parallel to the subspace $\{x_i = 0\}$. In other words,
there must be an $x^\beta\in G$ satisfying $\beta_i = \alpha_i$ and $\beta_j < \alpha_j$ for all $j \neq i$. Therefore $x^{\alpha-\beta}$ is a monomial of degree $a$ without $x_i$ 
and with positive exponents for each of the other variables. See Figure \ref{fig:witness-geometrically}. The number of such monomials is $m_{n-1}(a-n+1)$.  The relative interiors of the facets of $\Delta_\alpha$ are disjoint, 
so the events that a monomial appears in each one are independent. Additionally, $G$ must not contain any monomials that strongly divide $x^\alpha$, and the probability of this is 
$q^{m_n(a-n)}$ where $q = 1-p$.  Therefore, for $\alpha$ with $|\alpha| = D+a$ and $\alpha_i \geq a$ for all $i$,
 \begin{equation}\label{eq:prob-single-witness}
  \prob{w_\alpha} = \left( 1-q^{m_{n-1}(a-n+1)}\right)^n q^{m_n(a-n)}.
 \end{equation}
By linearity of expectation, a consequence of this formula is
\begin{equation}
\expect{W_a}= m_n(D+a-na)\left( 1-q^{m_{n-1}(a-n+1)}\right)^n q^{m_n(a-n)},
\end{equation}
because the number of exponent vectors $\alpha$ with $|\alpha| = D+a$ and $\alpha_i \geq a$ for all $i$ is $m_n(D+a-na)$.

\begin{lemma}\label{lemma:wa-bound}
Let $\alpha$ be an exponent vector with $a =|\alpha| - D \leq p^{-\frac{1}{n-1}}$ and $\alpha_i \geq a$ for all $i$.  Then,
\begin{equation}
 \frac{1}{2} p^n \left( m_{n-1}(a-n+1)\right)^n \leq \prob{w_\alpha} \leq p^n \left( m_{n-1}(a-n+1)\right)^n.
\end{equation}
\end{lemma}
\begin{proof}
 The union-bound implies that 
 \[ 1-q^{m_{n-1}(a-n+1)} \leq p m_{n-1}(a-n+1). \]
 The upper bound on $\prob{w_\alpha}$ follows from applying this inequality to the expression in (\ref{eq:prob-single-witness}). For the lower-bound, note that $\prob{w_\alpha}$ is 
bounded below by the probability that exactly one monomial is chosen to be in $G$ from the relative interior of each facet of $\Delta_\alpha$, and no other monomials 
are chosen in $\Delta_\alpha$.  The probability of this latter event is given by
  \[ p^n \left( m_{n-1}(a-n+1) \right)^n q^{m_n(a)-n} \]
 since there are $m_{n-1}(a-n+1)$ choices for the monomial picked in each facet.
 Now we use the fact that $m_n(a) \leq m_n(A) \leq p/2$ (and this is the reason for the choice of $A = \lfloor (p/2)^{-\frac{1}{n-1}}\rfloor - n$) to conclude
 \[ q^{m_n(a)-n} \geq 1-(m_n(a)-n)p \geq 1- \frac{(a+n)^{n-1}}{(n-1)!}p \geq \frac{1}{2}. \]
\end{proof}

\begin{lemma}\label{lemma:nbound-W}
 If $p \gg D^{-n+1}$ then
 $\displaystyle \lim_{D\to \infty} \expect{W} = \infty.$
\end{lemma}

\begin{proof}
If $\lim_{D\to \infty} p > 0$, then $\expect{W_{n-1}} \geq m_n(D-1)p^n$ which goes to infinity in $D$. Instead assume that $D^{-n+1} \ll p \ll 1$.
From Lemma \ref{lemma:wa-bound}, we have
\[ \prob{w_\alpha} \geq  \frac{1}{2} p^n \left( m_{n-1}(a-n+1) \right)^n \geq \frac{1}{2} p^n\left(\frac{(a-n)^{n-2}}{(n-2)!}\right)^n. \]
For $n-1 \leq a \leq A$ with $A = \lfloor (p/2)^{-\frac{1}{n-1}} \rfloor - n$,  one gets  $a \ll D$,  and hence
for $D$ sufficiently large, $na < D/2$, which means $D+a-na > D/2$.  Therefore 
\[ m_n(D+a-na) \geq \frac{D^{n-1}}{2^{n-1}(n-1)!}. \]
Since $m_n(D+a-na)$ is the number of exponent vectors $\alpha$ with $|\alpha| = D+a$ and $\alpha_i \geq a$ for all $i$, 
\[ \expect{W_a} = \sum_{\substack{|\alpha|=D+a \\ \alpha_i\geq a\, \forall i}} \prob{w_\alpha} \geq c_n D^{n-1}p^n(a-n)^{n(n-2)} \]
where $c_n > 0$ is a constant that depends only on $n$.  Summing up over $a$ gives the bound
\[ \expect{W} = \sum_{a = n-1}^A \expect{W_a} \geq c_n D^{n-1}p^n \sum_{a = n-1}^A(a-2n)^{n^2-2n}. \]
The function $f(A) = \sum_{a = n-1}^A (a-2n)^{n^2-2n}$ is polynomial in $A$ with lead term $t =  A^{n^2-2n+1}/(n^2-2n+1)$.  Since $A$ is proportional to $p^{-\frac{1}{n-1}}$, for $p$ sufficiently small $f(A) \geq t/2$ and so
\[ \expect{W} \geq c_n D^{n-1}p^n\frac{p^{-\frac{n^2-2n+1}{n-1}}}{2(n^2-2n+1)} = c'_n D^{n-1}p \]
and $D^{n-1}p$ goes to infinity as $D \to \infty$.
\end{proof}

\begin{lemma}\label{lemma:variance-W}
If $p \gg D^{-n+1}$ then
$$\lim_{D\to\infty}\frac{\var{W}}{\expect{W}^2}=0.$$
\end{lemma}
\begin{proof}
Since $W$ is a sum of indicator variables $w_\alpha$, we can bound $\var{W}$ by
\[ \var{W} \leq \expect{W} + \sum_{(\alpha,\beta)} \cov{w_\alpha,w_\beta}. \]
The covariance is easy to analyze in the following two cases.  If the degree of $\gcd(x^\alpha, x^\beta)$ is at most $D$, then $w_\alpha$ and $w_\beta$ depend 
on two sets of monomials being in $G$ which share at most one monomial. In this case $w_\alpha$ and $w_\beta$ are independent so $\cov{w_\alpha,w_\beta} = 0$.  
The second case is that $x^\alpha|x^\beta$ and $\alpha \neq \beta$.  If $w_\alpha = 1$, then $G$ contains a monomial that strictly divides $x^\beta$.  In this case $w_\alpha$ and 
$w_\beta$ are mutually exclusive, so $\cov{w_\alpha,w_\beta} < 0$. The cases with $\cov{w_\alpha,w_\beta}\leq 0$ are illustrated geometrically, for $n=3$, in Figure \ref{fig:no-covariance}.

\begin{figure}[h!]
\begin{subfigure}[t]{0.64\textwidth}
\includegraphics[width=0.48\textwidth]{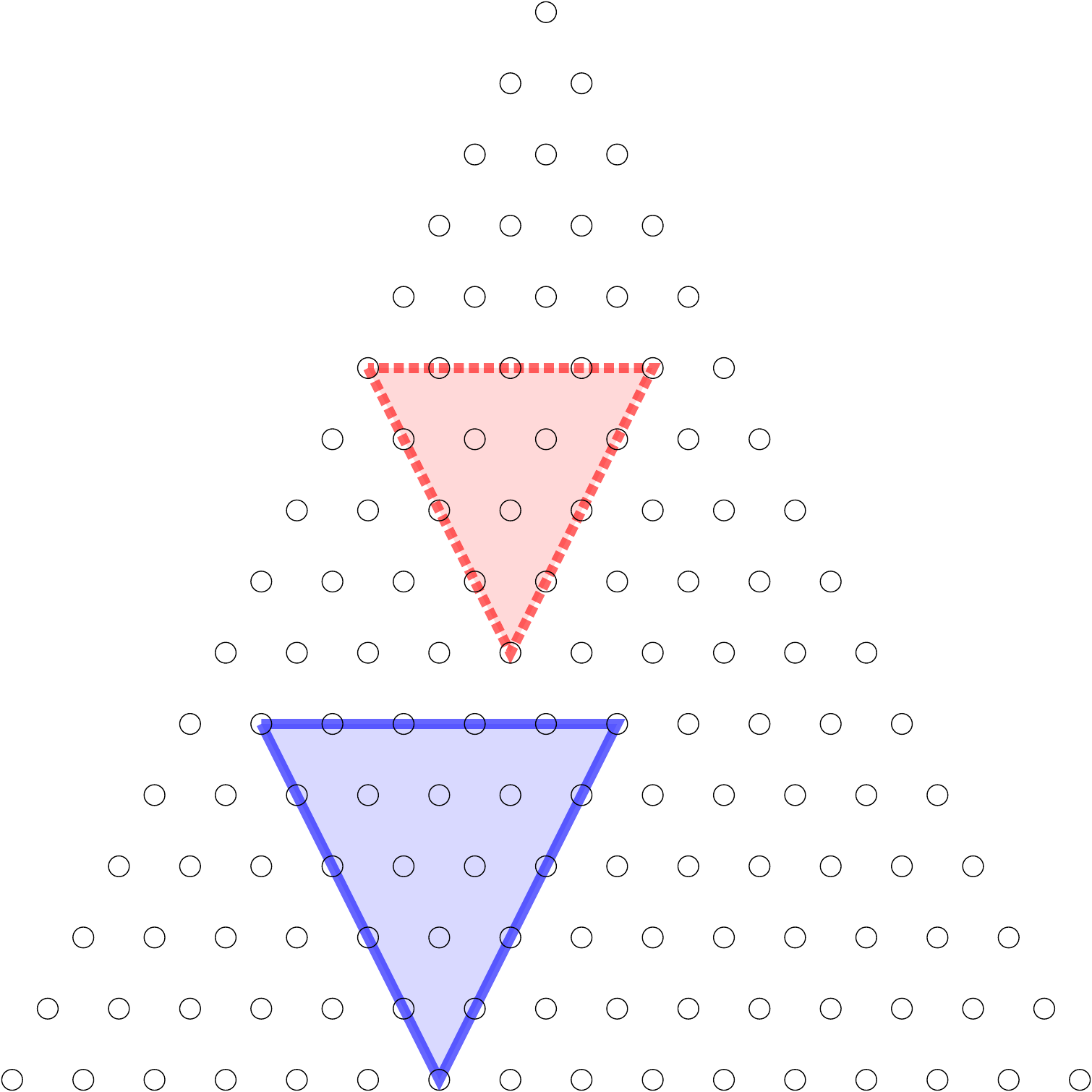}
\hfill
\includegraphics[width=0.48\textwidth]{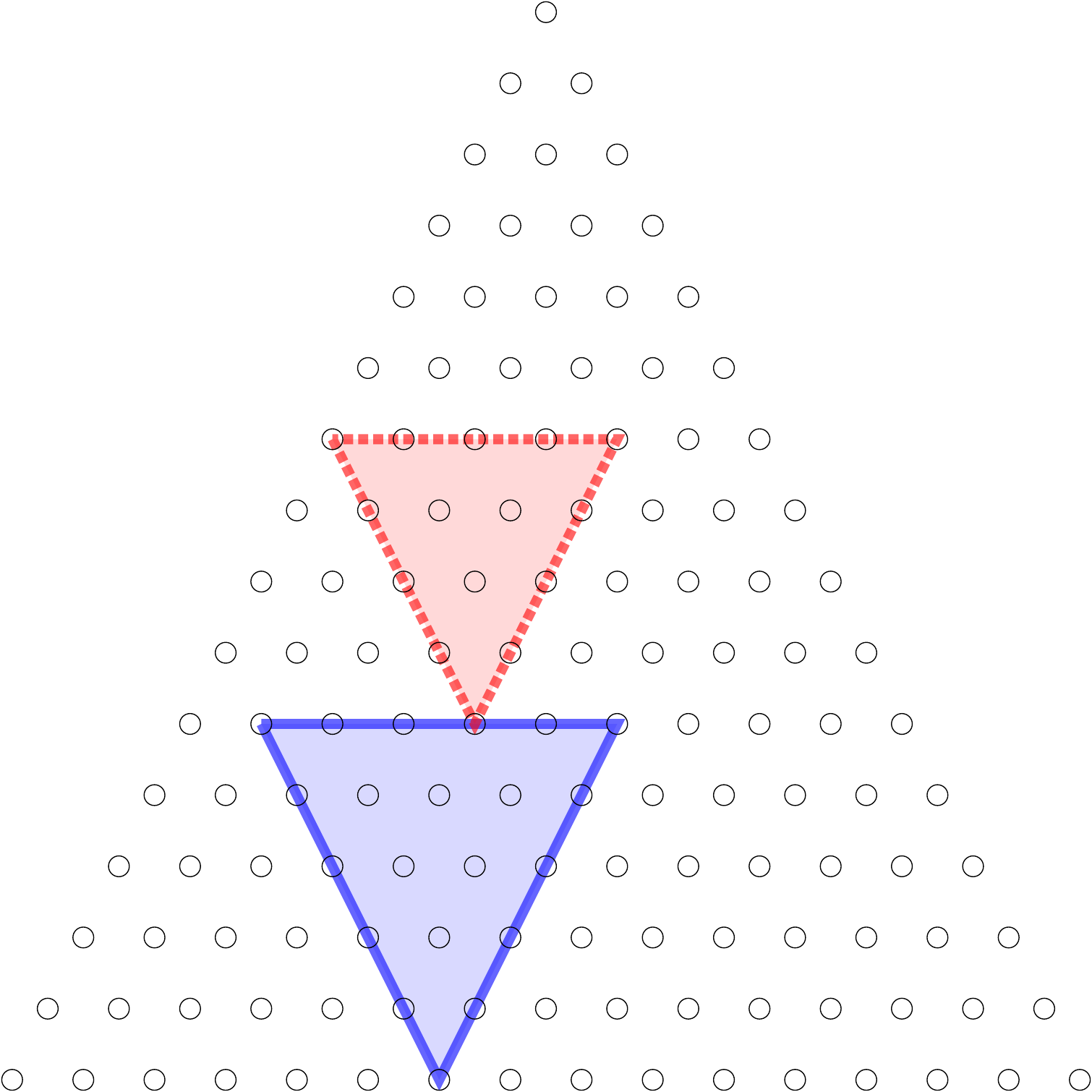}
\caption{If $\gcd(x^\alpha,x^\beta)$ has degree $\leq D$, then the intersection of $\Delta_\alpha$ (red/dotted) and $\Delta_\beta$ (blue/solid) is either empty or has cardinality 1. In either case, $\cov{w_\alpha,w_\beta}=0$.}
\end{subfigure}
\hfill
%
\begin{subfigure}[t]{0.31\textwidth}
\includegraphics[width=\textwidth]{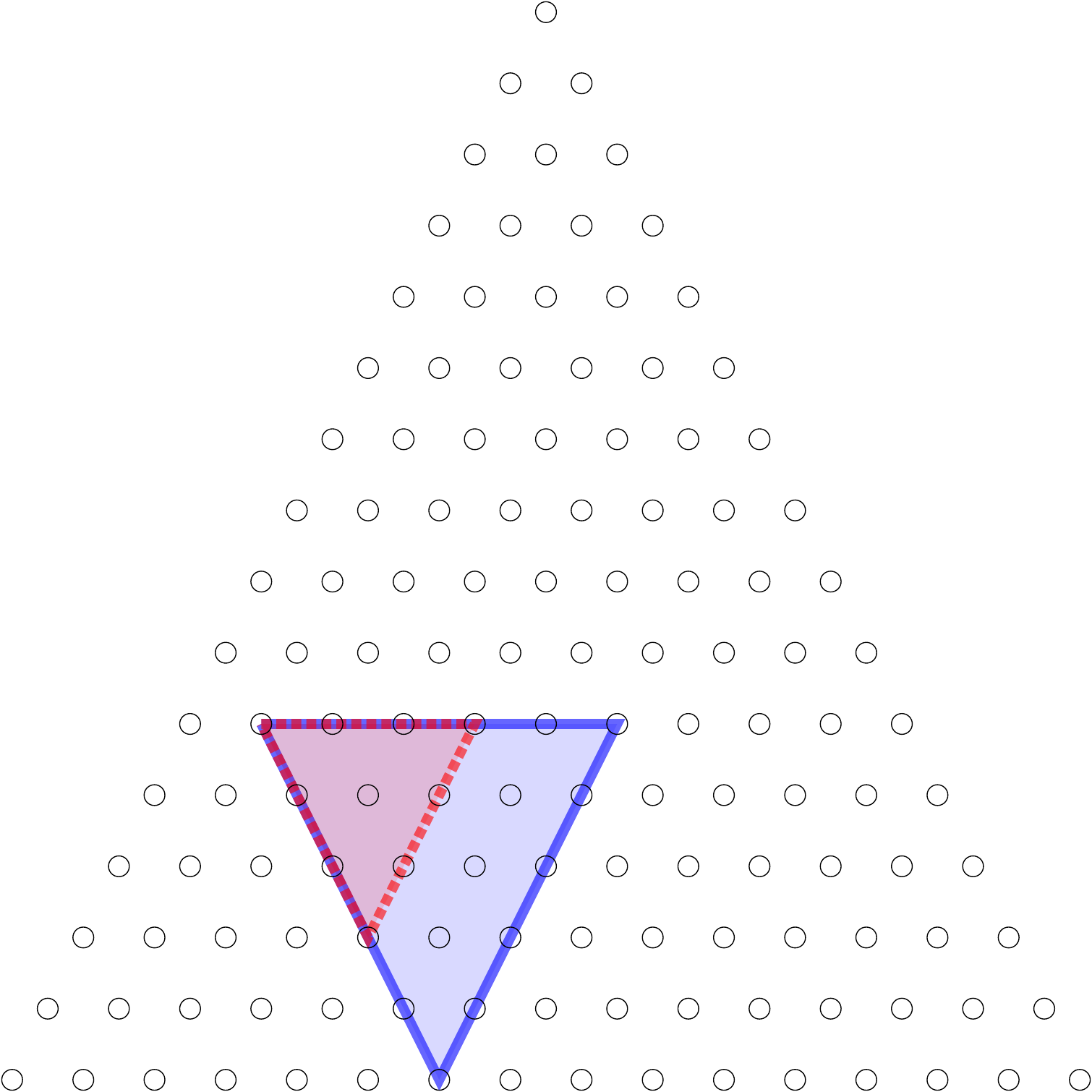}
\caption{If $x^\alpha|x^\beta$, then $\Delta_\alpha\sse\Delta_\beta$. In this case, $\cov{w_\alpha,w_\beta}<0$.}
\end{subfigure}
\caption{Pairs of witness lcm's with zero or negative covariance.}
\label{fig:no-covariance}
\end{figure}

Thus we focus on the remaining case, when $\deg \gcd(x^\alpha,x^\beta) > D$ and neither of $x^\alpha$ and $x^\beta$ divides the other.  
In other words $\Delta_\alpha$ and $\Delta_\beta$ have intersection of size $>1$ and neither is contained in the other.

Let $a = \deg(x^\alpha) - D$, $b = \deg(x^\alpha) - D$, which are the edge lengths of the simplices $\Delta_\alpha$ and $\Delta_\beta$ respectively.  
Let $c = \deg(\gcd(x^\alpha,x^\beta)) - D$, which is the edge length of the simplex  $\Delta_\alpha \cap \Delta_\beta$.  
Note that $0 < c < a$ due to assumptions made on $\alpha$ and $\beta$.  The number of common divisors of $x^\alpha$ and $x^\beta$ of degree $D$ is given by $m_n(c)$.  
Let $\delta_{\alpha,i}$ and $\delta_{\beta,i}$ denote the relative interiors of the $i$th facets of $\Delta_\alpha$ and $\Delta_\beta$,  respectively.  
The type of intersection of $\Delta_\alpha$ and $\Delta_\beta$ is characterized by signs of the entries of $\alpha-\beta$, which is described by a 3-coloring $C$ of $[n]$ with 
color classes $C_\alpha,C_\beta,C_\gamma$ for positive, negative, and zero, respectively.

Since $w_\alpha$ is a binary random variable,  $\cov{w_\alpha,w_\beta} = \prob{w_\alpha w_\beta} - \prob{w_\alpha}\prob{w_\beta}$, and hence it is bounded by  $\prob{w_\alpha w_\beta}$. 
Therefore we will focus on bounding this quantity.  Let $w_{\alpha,i}$ be the indicator variable for the event that $G$ contains a monomial $x_1^{u_1} \cdots x_n^{u_n}$
with $u_i = \alpha_i$ and $u_j < \alpha_j$ for each $j \neq i$.  Then \[ \prob{w_\alpha w_\beta} \leq \prob{\prod_{i=1}^n w_{\alpha,i}w_{\beta,i}}. \]
For $i \in C_\alpha$, the facet $\delta_{\alpha,i}$ does not intersect $\Delta_\beta$. See Figure \ref{fig:coloring1a}. For each $i \in C_\alpha$, we have
 \[ \prob{w_{\alpha,i}} = 1-q^{m_{n-1}(a-n+1)} \leq m_{n-1}(a-n+1)p \leq a^{n-2}p \leq A^{n-2}p \leq p^{1/(n-1)}. \]
Similarly for $i \in C_\beta$, $\prob{w_{\beta,i}} \leq p^{1/(n-1)}$.

For each pair $i \in C_\beta$ and $j \in C_\alpha$, facets $\delta_{\alpha,i}$ and $\delta_{\beta,j}$ intersect transversely.  Let $H$ be the bipartite graph on $C_\beta \cup C_\alpha$ formed by having $\{i,j\}$ as an edge if and only if there is a monomial in $G$ in $\delta_{\alpha,i} \cap \delta_{\beta,j}$.  Let $e_{i,j}$ be the event that $\{i,j\}$ is an edge of $H$.  Let $V$ denote the subset of $C_\beta \cup C_\alpha$ not covered by $H$.  If $w_\alpha w_\beta$ is true, then for each $i \in V \cap C_\beta$, there must be a monomial in $G$ in $\delta_{\alpha,i} \setminus \bigcup_{j \in C_\alpha} \delta_{\beta,j}$, and let $v_i$ be this event.  Similarly for each $j \in V \cap C_\alpha$, there must be a monomial in $G$ in $\delta_{\beta,j} \setminus \bigcup_{i \in C_\beta} \delta_{\alpha,i}$, and let $v_j$ be this event. See Figure \ref{fig:colorings} for the geometric intuition behind these definitions.

Note that all events $e_{i,j}$ and $v_i$ are independent since they involve disjoint sets of variables.  Therefore
 \[ \prob{\prod_{i \in C_\alpha}w_{\alpha,i}\prod_{i \in C_\beta}w_{\beta,i}} \leq \sum_H \prod_{\{i,j\} \in E(H)} \prob{e_{i,j}} \prod_{i \in V}\prob{v_i}. \]
For any $(i,j) \in C_\beta \times C_\alpha$,
\[ |\delta_{\alpha,i} \cap \delta_{\beta,j}| \leq m_{n-2}(c) \leq c^{n-3} \leq p^{-\frac{n-3}{n-1}}. \]
Therefore
\[ \prob{e_{i,j}} = 1 - q^{|\delta_{\alpha,i} \cap \delta_{\beta,j}|} \leq p|\delta_{\alpha,i} \cap \delta_{\beta,j}| \leq p^{\frac{2}{n-1}}. \]
We also know that for $i \in C_\beta$, $\prob{v_i} \leq \prob{w_{\alpha,i}} \leq p^{1/(n-1)}$, and similarly for $i \in C_\alpha$.  So then
\[ \sum_H \prod_{\{i,j\} \in E(H)} \prob{e_{i,j}} \prod_{i \in V}\prob{v_i} \leq \sum_H p^{\frac{2|E(H)| + |V|}{n-1}}. \]
The number of graphs $H$ is $2^{|C_\beta||C_\alpha|} \leq 2^{n^2}$ and for any graph $H$, $2|E(H)| + |V| \geq |C_\beta|+|C_\alpha|$ since every element of $C_\beta \cup C_\alpha$ must be covered by $H$ or in $V$.  Then
\[ \prob{\prod_{i \in C_\alpha}w_{\alpha,i}\prod_{i \in C_\beta}w_{\beta,i}} \leq 2^{n^2} p^{\frac{|C_\beta|+|C_\alpha|}{n-1}}. \]

\begin{figure}
\centering
\begin{subfigure}[b]{0.46\textwidth}
\centering
\includegraphics[width=\textwidth]{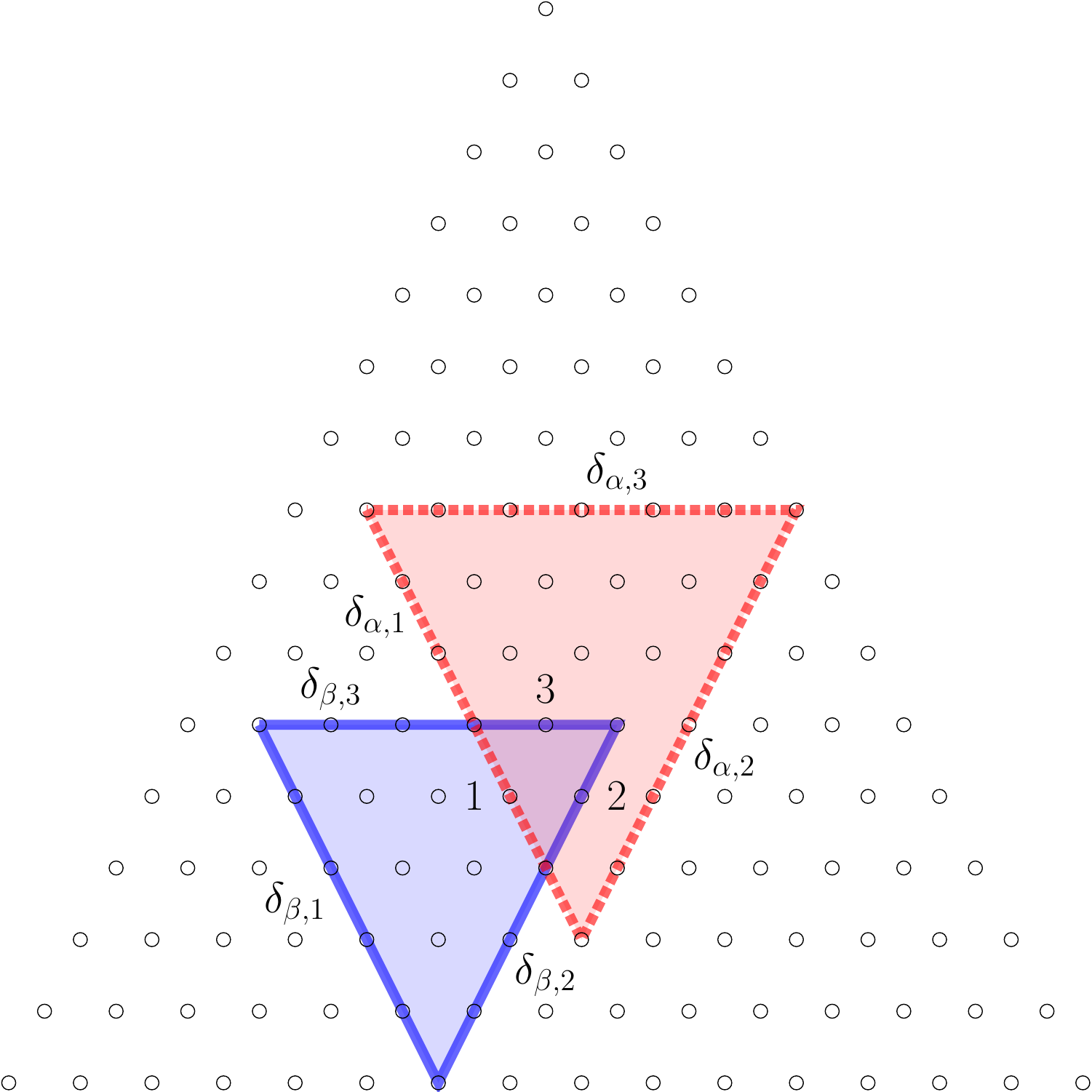}
\caption{An intersection of $\Delta_\alpha$ (red/dotted) and $\Delta_\beta$ (blue/solid). The facets of the intersection are labeled $1,2,3$, and 
the coloring of $[3]$ associated with this intersection is $(-,+,+)$; equivalently $C_\alpha=\{2,3\}$, $C_\beta=\{1\}$ and $C_\gamma=\emptyset$. Since $1\in C_\beta$, the facet 
$\delta_{\beta,1}$ does not intersect $\Delta_\alpha$. Similarly, since $C_\alpha=\{2,3\}$, the facets $\delta_{\alpha,2}$ and $\delta_{\alpha,3}$ do not intersect $\Delta_\beta$.}
\label{fig:coloring1a}
\end{subfigure}
\hfill
\begin{subfigure}[b]{0.47\textwidth}
\centering
\includegraphics[width=0.8\textwidth]{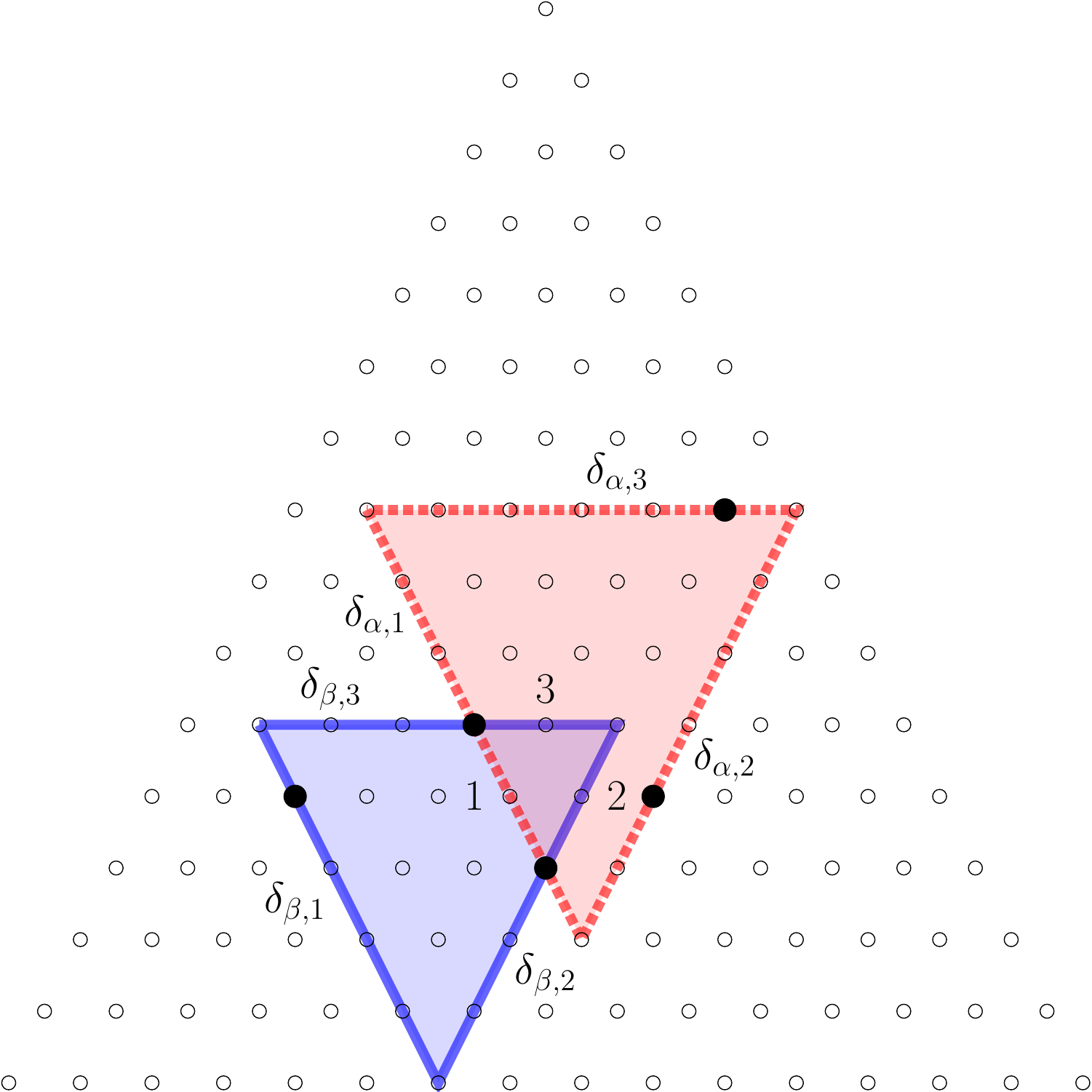}
\caption{A set of five generators (above, in black), for which $w_\alpha w_\beta=1$. Since one generator belongs to the intersection of facets 1 and 3, 
the associated bipartite graph $H$ (below) has edge $\{1,3\}$. Here $V=\{2\}$, indicating that $G$ must contain a generator in 
$\delta_{\beta,2}\backslash(\delta_{\alpha,1}\cup\delta_{\alpha,3})$. }

\bigskip

\includegraphics[width=0.4\textwidth]{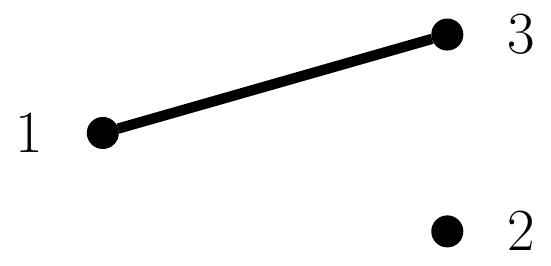}
\end{subfigure}

\caption{An illustration of intersection types, color classes, the graph $H$, and the set $V$.}
\label{fig:colorings}
\end{figure}


Finally for each $i \in C_\gamma$, facets $\delta_{\alpha,i}$ and $\delta_{\beta,i}$ have full dimensional intersection.  Again $G$ may contain distinct monomials in 
$\delta_{\alpha,i}$ and $\delta_{\beta,i}$, or just one in their intersection.  However, $\delta_{\alpha,i}$ does not intersect any other facets of $\Delta_\beta$ so there are only two cases.
\[ \prob{w_{\alpha,i}w_{\beta,i}} \leq (1-q^{m_{n-1}(a-n+1)})^2 + 1-q^{m_{n-1}(c-n+1)} \leq p^{2/(n-1)} + p^{1/(n-1)} \leq 2p^{1/(n-1)}. \]

Combining these results, we have
$$ \prob{w_\alpha w_\beta} \leq 2^{n^2}p^{\frac{|C_\beta|+|C_\alpha|}{n-1}} \prod_{i \in C_\alpha} p^{\frac{1}{n-1}} \prod_{j \in C_\beta} p^{\frac{1}{n-1}} \prod_{i \in C_\gamma}2p^{\frac{1}{n-1}} 
 \leq 2^{n^2+|C_\gamma|} p^{\frac{2n-|C_\gamma|}{n-1}}. $$

To sum up over all pairs $\alpha,\beta$ with potentially positive variance, we must count the number of pairs of each coloring $C$.  To do so, first fix $C$ and $\alpha$ and 
count the number of $\beta$ such that the intersection of $\Delta_\alpha$ and $\Delta_\beta$ have type $C$.  Note that the signs of the entries of $\alpha-\beta$ are prescribed, 
and that the entries of $\alpha-\beta$ are bounded by $p^{-\frac{1}{n-1}}$ because the degrees of $x^\alpha$ and $x^\beta$ are each within $p^{-\frac{1}{n-1}}$ of the degree of their gcd.  
A rough bound then on the number of values of $\beta$ is $(p^{-\frac{1}{n-1}})^{n-|C_\gamma|}$.  The number of values of $\alpha$ for each choice of $a$ is $m_n(D+a-na) \leq D^{n-1}$, 
so summing over all possible values of $a$, the number of $\alpha$ values is bounded by $p^{-\frac{1}{n-1}}D^{n-1}$.  Therefore
$$ \sum_{(\alpha,\beta) \text{ of type } C} \cov{w_\alpha,w_\beta} \leq \#\{(\alpha,\beta) \text{ of type } C\} 2^{n^2+|C_\gamma|} p^{\frac{2n-|C_\gamma|}{n-1}}  $$
$$ \leq p^{-\frac{1}{n-1}}D^{n-1} (p^{-\frac{1}{n-1}})^{n-|C_\gamma|}  2^{n^2+|C_\gamma|} p^{\frac{2n-|C_\gamma|}{n-1}} \leq 2^{n^2+n} D^{n-1}p \leq \frac{2^{n^2+n}}{c'_n}\expect{W}. $$
Then summing over all colorings $C$, of which there are less than $3^n$, shows that $\var{W} \leq c''_n\expect{W}$ for $c''_n >0$ depending only on $n$.
Therefore
 \[ \lim_{D\to \infty} \frac{\var{W}}{\expect{W}^2} \leq \lim_{D\to \infty} \frac{c''_n}{\expect{W}} = 0. \]
\end{proof}

\begin{proof} [Proof of Theorem \ref{thm:proj-dim}]
If $ p \ll D^{-n+1}$, Proposition \ref{prop:pdimzero} implies that $\pdim(S/M) = 0$. If $p \gg D^{-n+1}$, Lemma \ref{lemma:nbound-W} proves that $\expect{W} \to \infty$ as
$D \to \infty$. Since  Lemma \ref{lemma:variance-W} shows that $\prob{W>0} \to 1$, we conclude that there is a witness set asymptotically almost surely. This is equivalent
to $\pdim(S/M) =n$.
\end{proof}

\subsection{Consequences of Theorem \ref{thm:proj-dim}}


An $S$-module $S/M$ is called \emph{Cohen-Macaulay} if $\dim(S/M)=depth(S/M)$. Since $S$ is a polynomial ring, this condition is equivalent to $\dim(S/M)=n-\pdim(S/M)$, by the Auslander-Buchsbaum theorem \cite[Corollary 19.10]{Eisenbudbook}. From Theorem \ref{thm:proj-dim} we obtain the proof of the Cohen-Macaulayness result announced in the introduction.

\begin{proof} [Proof of Corollary \ref{cor:CM}]
For a monomial ideal $M\sse S$, the Krull dimension of $S/M$ is zero if and only if for each $i=1,\ldots,n$, $M$ contains a minimal generator of the form $x_i^j$ for $j=1,\ldots,n$. For $M\sim\gradedmodel$, this can only occur if every monomial in the set $\{x_1^D,x_2^D,\ldots,x_n^D\}$ is chosen as a minimal generator, an event that has probability $p^n$. Thus for fixed $n$ and $p\ll1$, $\prob{\dim(S/M)=0}=p^n\to 0$
as $D\to\infty$. If also $D^{-n+1}\ll p$, then by Theorem \ref{thm:proj-dim}, $\prob{\pdim(S/M)=n}\to 1$. Together, these imply that $\prob{S/M \text{ is Cohen-Macaulay}}\to 0$ as $D\to\infty$.
\end{proof}
Our probabilistic result on Cohen-Macaulayness is an interesting companion to a recent result of Erman and Yang. In \cite{ermanFLAG}, they consider random squarefree monomial ideals in $n$ variables, defined as the Stanley-Reisner ideals of random flag complexes on $n$ vertices, and study their asymptotic behavior as $n\to\infty$. Though the model is very different, they find a similar result: for many choices of their model parameter, Cohen-Macaulayness essentially never occurs.


Our second corollary is about Betti numbers. By the results of Brun and R\"{o}mer \cite{brun2004betti}, which extended those of Charalambous \cite{charalambous1991betti} (see also \cite{boocher2018lower}), a monomial ideal with projective dimension $d$ will satisfy $\beta_i(S/M)\ge{d\choose i}$ for all $1\le i\le d$. In the special case $d=n$, Alesandroni gives a combinatorial proof of the implied inequality $\sum_{i=0}^{n}\beta_i(S/M)\geq 2^n$ \cite{alesandroniPDIM}. These inequalities are of interest because they relate to the long-standing \emph{Buchsbaum-Eisenbud-Horrocks conjecture} \cite{buchsbaum+eisenbud,hartshornePROBLEMS}, that $\beta_i(N)\ge {c\choose i}$ for $N$ an $S$-module of codimension $c$. In 2017, Walker \cite{walkerBEH} settled the BEH conjecture outside of the characteristic 2 case. Here we show that a probabilistic result, which holds regardless of characteristic, follows easily from Theorem \ref{thm:proj-dim}.

\begin{cor}\label{cor:BEH}
	Let  $M\sim\gradedmodel$ and $p=p(D)$. If $D^{-n+1} \ll p$, then asymptotically almost surely $\beta_i(S/M)\geq {n\choose i}$ for all $1\le i\le n$.
\end{cor}
\begin{proof}
	Follows immediately from \cite[Theorem 1.1]{brun2004betti} and Theorem \ref{thm:proj-dim}.
\end{proof}

\section{Genericity and  Scarf monomial ideals}\label{sec:gen-scarf}


Let $M = \langle G \rangle$ be a monomial ideal with minimal generating set $G=\{g_1,\ldots,g_r\}$. For each subset $I$ of $\{1,\ldots,r\}$ let $m_I=\lcm(g_i\mid i\in I)$. 
Let $a_I\in\N^n$ be the exponent vector of $m_I$ and let $S(-a_I)$ be the free $S$-module with one generator in multidegree $a_I$. The \emph{Taylor complex} of $S/M$ is the $\Z^n$-graded module 
$$\mathcal{F}=\bigoplus_{I\sse\{1,\ldots,r\}}S(-a_I)$$ with basis denoted by $\{e_I\}_{I\sse\{1,\ldots,r\}}$, and equipped with the differential:
  $$
  d(e_I)=
  \sum_{i\in I}sign(i,I)\cdot\frac{m_I}{m_{I\backslash i}}\cdot e_{I\backslash i},
  $$
  where $sign(i,I)$ is $(-1)^{j+1}$ if $i$ is the $j$th element in the ordering of $I$. This is a free resolution of $S/M$ over $S$ with $2^r$ terms; the terms are in bijection with the $2^r$ subsets of $G$, and the term corresponding to $I\subseteq G$ appears in homological degree $|I|$. 
The \textit{Scarf complex} of $M$, written $\Delta_M$, is a simplicial complex on the vertex set $\{1,\ldots,r\}$. Its faces are defined by
$$
\Delta_M =
\{I\sse \{1,\ldots,r\}\mid m_I\neq m_J\text{ for all }J\sse \{1,\ldots,r\}, J\neq I\}.$$
The \emph{algebraic Scarf complex} of $M$, written $\mathcal{F}_{\Delta_M}$, is defined as the subcomplex of the Taylor complex that is supported on $\Delta_M$.
The algebraic Scarf complex $\mathcal{F}_{\Delta_M}$ is a subcomplex of every free resolution of $S/M$, in particular of every minimal free resolution  \cite[Section 6.2]{millerCCA}. 
When $\mathcal{F}_{\Delta_M}$ is a minimal free resolution of $S/M$, we say that $M$ \emph{is Scarf.}

A sufficient condition for $M$ to be Scarf is genericity. A monomial ideal $M$ is \emph{strongly generic} if no variable $x_i$ appears with the same nonzero exponent in two distinct minimal generators of $M$. 
In \cite{bayerMONOMIALRES}, Bayer, Peeva and Sturmfels proved that strongly generic monomial ideals are Scarf. (Note that the authors used the term \emph{generic} for what is now called \emph{strongly generic}.)

Miller and Sturmfels defined a less restrictive notion of genericity in \cite{millerCCA}. 
A monomial ideal $M$ is \emph{generic} if whenever two distinct minimal generators $g_i$ and $g_j$ have the same positive degree in some variable, a third generator $g_k$ strongly divides $\lcm(m_i,m_j)$. Monomial ideals that are generic in this broader sense are also always Scarf. 

\subsection{Genericity of random monomial ideals}

Since every monomial ideal in this paper is generated in degree $D$, $M$ is generic if and only if it is strongly generic, and these are characterized by the property that 
for every distinct pair of monomials $x^\alpha$ and $x^\beta$ in $G$, either $\alpha_i = 0$ or $\alpha_i \neq \beta_i$ for all $i = 1,\ldots,n$. Now we prove the threshold
theorem about the genericity of random monomial ideals.

\begin{proof}[Proof of Theorem \ref{thm:generic}]
 Let $V$ be the indicator variable that $M$ is strongly generic.  For each variable $x_i$ and each exponent $c$, let $v_{i,c}$ denote the indicator variable for the event that there is at most one monomial in $G$ with $x_i$ exponent equal to $c$, and let $V_i = \prod_{c=1}^D v_{i,c}$.  Then
\[ V = \prod_{i = 1}^n V_i. \]

Given a set $\Gamma$ of monomials of degree $D$ in $S$ with $|\Gamma| = m$, the probability that $G$ contains at most one monomial in $\Gamma$ is 
 $$ \prob{|\Gamma \cap G| \leq 1} = q^m + mpq^{m-1} \geq 1 - mp + mp(1-(m-1)p) \geq 1 - m^2p^2. $$
On the other hand
 $$ \prob{|\Gamma \cap G| \leq 1} \leq \prob{|\Gamma \cap G| \neq 2} = 1 - \binom{m}{2}p^2q^{m-2}. $$
Assuming that $p \ll m^{-1}$ then for $p$ sufficiently small, $q^{m-2} \geq 1/2$ so
\begin{equation} \label{eq:upperboundgeneric}
  \prob{|\Gamma \cap G| \leq 1} \leq 1 - \frac{(m-1)^2}{4}p^2. 
\end{equation}
The above gives bounds on $\prob{v_{i,c}}$ by taking $\Gamma$ to be the set of monomials of degree $D$ with $x_i$ degree equal to $c$.  
Then $|\Gamma| = m_{n-1}(D-c) \leq D^{n-2}$, hence
 \[ \prob{v_{i,c}} \geq 1 - D^{2n-4}p^2. \]
By the union-bound,
 \[ \prob{V} \geq 1 - \sum_{i = 1}^n\sum_{c=1}^D  (1 - \prob{v_{i,c}}) \geq 1 - np^2D^{2n-3}. \]
Therefore, for $p \ll D^{-n+3/2}$, $\prob{V}$ goes to 1.

For a lower bound on $\prob{V_i}$, let $U_i$ be the random variable that counts the number of values of $c$ for which $v_{i,c}$ is false.  Assuming that $p \ll D^{-n+2}$ and $p$ sufficiently small, and using the upper bound on $\prob{v_{i,c}}$ established in (\ref{eq:upperboundgeneric}), we get
\[ \expect{U_i} = \sum_{c=1}^D (1 - \prob{v_{i,c}}) \geq  \frac{p^2}{4} \sum_{c=1}^D (m_{n-1}(D-c)-1)^2. \]
The function $f(D) = \sum_{c=1}^D (m_{n-1}(D-c)-1)^2$ is a polynomial in $D$ with lead term $t = D^{2n-3}/(n-2)!^2(2n-3)$.  Thus for $D$ sufficiently large, $f(D) \geq t/2$ so
 \[ \expect{U_i} \geq \frac{p^2D^{2n-3}}{8(n-2)!^2(2n-3)}. \]
Therefore, for $D^{-n+3/2} \ll p \ll D^{-n+2}$,
 \[ \lim_{D\to \infty} \expect{U_i} = \infty. \]
Since the indicator variables $v_{i,1},\ldots,v_{i,D}$ are independent, $\var{U_i} \leq \expect{U_i}$.  By the second moment method
\[ 0 = \lim_{D\to \infty} \prob{U_i = 0} = \lim_{D\to \infty} \prob{V_i} \geq \lim_{D\to \infty} \prob{V}. \]

Finally, note that for $D$ fixed, $\prob{V}$ is monotonically decreasing in $p$.  Therefore $\prob{V}$ goes to 0 as $D$ goes to infinity for all $p \gg D^{-n+3/2}$.
\end{proof}

\subsection{Scarf complexes of random monomial ideals}\label{subsec:scarf-proof}

The main result of this subsection is Theorem \ref{thm:scarf}: as $D\to\infty$, $M$ is almost never Scarf when $p$ grows faster 
than $D^{-n+2-1/n}$.   We also know that $M$ is almost never Scarf when $p$ grows slower than $D^{-n+1}$ for the trivial reason 
that the ideal is usually empty.  This leaves a gap where we do not know the asymptotic behavior.

The logic of this proof is as follows: 
suppose that $L\sse G$ is a witness set to $\pdim(S/M)=n$. By Theorem \ref{thm:pdim-equiv-dom-set-strong}, the free module $S(-a_L)$ appears in the minimal free 
resolution of $S/M$ in homological degree $n$. 
Suppose further that there exists $g\in G\backslash L$, such that $g$ divides $\lcm(L)$. Then $\lcm(L)=\lcm(L\cup \{g\})$, so by definition $S(-a_L)$ does \textit{not} appear in the Scarf complex of $M$. Thus, the minimal free resolution strictly contains the Scarf complex, and $M$ is not Scarf. When this occurs, we call $L\cup \{g\}$ a \emph{non-Scarf witness set}. We now show that for $p\gg D^{-n+2-1/n}$, the number of non-Scarf witness sets is a.a.s.\ positive.

For each $x^\alpha\in S$, define $y_\alpha$ as the indicator random variable:
$$
y_\alpha=
\begin{cases} 
1 & x^\alpha \text{ is the lcm of a non-Scarf witness set}\\
0 & \text{ otherwise. }
\end{cases}
$$
For each integer $a\geq 1$, define the random variable $Y_a$ that counts the monomials of degree $D+a$ that are lcm's of non-Scarf witness sets. Let $Y$ be the sum of these variables over a certain range of $a$:
$$
Y_a
=
\sum_{\substack{|\alpha|=D+a \\ \alpha_i\geq a\, \forall i}} y_\alpha,
\hspace*{6em}
Y=
\sum_{a=2}^A
Y_a
$$
where $A = {\lfloor (p/2)^{-\frac{1}{n-1}}\rfloor} - n$.

For $y_\alpha$ to be true, there must be a monomial in $G$ in the relative interior of each facet of the simplex $\Delta_\alpha$ and one of the facets must have at least two monomials 
in $G$.  Additionally $G$ must have no monomials in the interior of $\Delta_\alpha$.  For $x^\alpha\in S$ with $|\alpha|=D+a$, and $\alpha_i\geq a$ for $i=1,\ldots,n$,
 \begin{equation}\label{eq:prob-single-nonscarf-witness}
 \prob{Y_a}= m_n(D+a-na)\left(\left( 1-q^{m_{n-1}(a-n+1)}\right)^n - \left(m_{n-1}(a-n+1)pq^{m_{n-1}(a-n+1)-1}\right)^n\right) q^{m_n(a-n)}.
 \end{equation}
This follows from the same argument as the formula \ref{eq:prob-single-witness}, subtracting the case that exactly one monomial lies on each facet. The relevant bound is
\begin{lemma}\label{lemma:va-bound}
Let $\alpha$ be an exponent vector with $a =|\alpha| - D \leq p^{-\frac{1}{n-1}}$ and $\alpha_i \geq a$ for all $i$.  Then,
\begin{equation}
 \frac{1}{4} p^{n+1} m_{n-1}(a-n+1)^{n+1} \leq \prob{y_\alpha} \leq \frac{1}{2}p^{n+1} m_{n-1}(a-n+1)^{n+1}.
\end{equation}
\end{lemma}
\begin{proof}
 The union-bound implies that 
 \[ 1-q^{m_{n-1}(a-n+1)} \leq p m_{n-1}(a-n+1). \]
 The upper bound on $\prob{y_\alpha}$ follows from applying this inequality to the expression in equation \ref{eq:prob-single-witness}.
 
 For the lower-bound, note that $\prob{y_\alpha}$ is bounded below by the probability that exactly two monomials are chosen to be in $G$ 
from the relative interior of one of the facets of $\Delta_\alpha$ and exactly one is chosen from each other facet, and no other monomials are chosen in $\Delta_\alpha$.  
The probability of this event is given by
  $$ \binom{m_n(a-n)}{2}m_n(a-n)^{n-1}p^{n+1} q^{m_n(a)-n-1} $$
 since there are $m_n(a-n)$ choices for the monomial chosen in each facet.
 Also by the union-bound we have
 $$ q^{m_n(a)-n-1} \geq 1-(m_n(a)-n-1)p \geq 1- \frac{(a+n)^{n-1}}{(n-1)!}p \geq \frac{1}{2}. $$
\end{proof}

We can then find a threshold for $p$ where non-Scarf witness sets are expected to appear frequently.

\begin{lemma}\label{lemma:nbound-Y}
 If $D^{-n+2-1/n} \ll p$ then
 $\displaystyle \lim_{D\to \infty} \expect{Y} = \infty. $
\end{lemma}

\begin{proof}
We follow the same argument as in the proof of Lemma \ref{lemma:nbound-W}.
If $\lim_{D\to \infty} p > 0$, then $\expect{Y_{n}} \geq m_n(D-2)p^{n+1}q$ which goes to infinity in $D$. Instead assume 
that $D^{-n+2-1/n} \ll p \ll 1$ and take $n-1 \leq a \leq p^{-\frac{1}{n-1}}$.
As in the proof of Lemma \ref{lemma:nbound-W}, for $D$ sufficiently large
$$ m_n(D+a-na) \geq \frac{D^{n-1}}{2^{n-1}(n-1)!}. $$
Therefore
$$ \expect{Y_a} \geq c_n D^{n-1}p^{n+1}a^{(n+1)(n-2)} $$
where $c_n > 0$ is constants that depends only on $n$.  Summing up over $a$ gives the bound
$$ \expect{Y} \geq c'_n D^{n-1}p^{\frac{n}{n-1}} $$
and $D^{n-1}p^{\frac{n}{n-1}}$ goes to infinity as $D \to \infty$.
\end{proof}

\begin{lemma}\label{lemma:variance-Y}
If $p \gg D^{-n+2-1/n}$ then
$$\lim_{D\to\infty}\frac{\var{Y}}{\expect{Y}^2}=0.$$
\end{lemma}

\begin{proof}
 The proof follows the same structure as that of Lemma \ref{lemma:variance-W}.  We bound $\var{Y}$ by
 \[ \var{Y} \leq \expect{V} + \sum_{(\alpha,\beta)} \cov{y_\alpha,y_\beta}. \]
 For the pair of exponent vectors $(\alpha,\beta)$, $y_\alpha$ and $y_\beta$ are independent or mutually exclusive in the same set of cases as for $w_\alpha$ and $w_\beta$, 
in which case $\cov{y_\alpha,y_\beta}$ is non-positive.  The remaining case is  when the simplices $\Delta_\alpha$ and $\Delta_\beta$ intersect and neither is contained in the other.  
Let $C = (C_\alpha,C_\beta,C_\gamma)$ be the coloring corresponding to this pair.

 Define indicators $e_i$, $v_{i,j}$ and graph $H$ as in the proof of Lemma \ref{lemma:variance-W}.  It was shown that $\prob{w_\alpha w_\beta}$ is bounded above by
  \[ B = 2^{n^2+|C_\gamma|} p^{\frac{2n-|C_\gamma|}{n-1}}. \]
For $y_\alpha y_\beta$ to be true, it must be that $w_\alpha w_\beta$ is true, plus an extra monomial appears in some facet of $\Delta_\alpha$ and the same for $\Delta_\beta$.
 We will enumerate the cases of how this can occur, and modify the bound $B$ in each case to give a bound on $\prob{y_\alpha y_\beta}$.
 Recall that for a set $\Gamma$ of size $m$, we have that the probability of at least $2$ monomials in $G$ being chosen from $\Gamma$ is bounded
  \[ \prob{| \Gamma \cap G| \geq 2} \leq m^2p^2. \]
 There are two cases where a single monomial in $G$ is the extra one for both $y_\alpha$ and $y_\beta$:
 \begin{itemize}
  \item For some $i \in C_\gamma$, there are at least two monomials in $\delta_{\alpha,i} \cap \delta_{\beta,i}$.  The probability that this occurs is bounded by 
   $ m_{n-1}(A)^2p^2 \leq p^{\frac{2}{n-1}} $
  and this replaces a factor in the original bound $B$ of $p^{\frac{1}{n-1}}$, so the probability of $y_\alpha y_\beta$ being true and this occurring for some fixed choice of $i$ is bounded by $Bp^{\frac{1}{n-1}}$.
  \item For some edge $(i,j)$ of $H$, there are at least two monomials in $\delta_{\alpha,i} \cap \delta_{\beta,j}$.  The probability that this occurs is bounded by
   $ m_{n-2}(A)^2p^2 \leq p^{\frac{4}{n-1}} $
  and this replaces a factor in $B$ of $p^{\frac{2}{n-1}}$.
 \end{itemize}
 In the rest of the cases the extra monomial for $v_\alpha$ is distinct from the extra one for $v_\beta$. For $v_\alpha v_\beta$ to be true, two of these cases must be paired.  
We describe the situation for $v_\alpha$, but the $v_\beta$ case is symmetric.
 \begin{itemize}
  \item For some $i \in C_\beta$, the vertex in the graph $H$ has degree at least 2.  In this case $2|E(H)| + |V| \geq |C_\alpha| + |C_\beta| + 1$, one greater than the bound in the original computation of $B$.  Thus we pick up an extra factor of $p^{\frac{1}{n-1}}$ over $B$.
  \item For $i \in C_\alpha$ or $i \in C_\beta \cap V$ or $i \in C_w$ with no monomial in $\delta_{\alpha,i}\cap \delta_{\beta,i}$, there are at least two monomials in $\delta_{\alpha,i} \setminus \bigcup_j \delta_{\beta,j}$.  We replace a factor  of $p^{\frac{1}{n-1}}$ in $B$ by $p^{\frac{2}{n-1}}$.
  \item For $i \in C_\beta \setminus V$ or $i \in C_w$ with a monomial in $\delta_{\alpha,i}\cap \delta_{\beta,i}$, there is a monomial in $\delta_{\alpha,i} \setminus \bigcup_j \delta_{\beta,j}$.  Thus in the bound we pick up an extra factor of $p^{\frac{1}{n-1}}$ over $B$.
 \end{itemize}

 The probability of the first case being true is bounded by is $Bp^{\frac{1}{n-1}}$, while in all others it is bounded by $Bp^{\frac{2}{n-1}}$, and the former bound dominates.
 The total number of cases among all the situations above is some finite $N$ (depending only on $n$) so we can conclude that
 \[ \prob{y_\alpha y_\beta} \leq NBp^{\frac{1}{n-1}}. \]
 The remainder of the proof is identical to that of Lemma \ref{lemma:variance-W}, and so we arrive at
 $$ \var{Y} \leq N2^{n^2+n} D^{n-1}p^{\frac{n}{n-1}} \leq \frac{N2^{n^2+n}}{c'_n}\expect{Y}, $$
and therefore
$$ \lim_{D\to \infty} \frac{\var{Y}}{\expect{Y}^2} \leq \lim_{D\to \infty} \frac{c''_n}{\expect{Y}} = 0.$$ \end{proof}

\begin{proof} [Proof of Theorem \ref{thm:scarf}]
If $p \gg D^{-n+2-1/n}$, Lemma \ref{lemma:nbound-Y} proves that $\expect{Y} \to \infty$ as
$D \to \infty$. By the second moment method, Lemma \ref{lemma:variance-Y} implies that $\prob{Y>0} \to 1$.  We conclude that there is a non-Scarf witness set asymptotically almost surely, in which case $M$ is not Scarf.
\end{proof}
\section{Trends in the average Betti numbers of monomial resolutions}

For a (strongly) generic monomial ideal in $S = k[x_1,\ldots, x_n]$ with $r$ minimal generators, the Scarf complex is a subcomplex of the boundary of an $n$-dimensional simplicial
polytope with $r$ vertices where at least one facet has been removed \cite[Proposition 5.3]{bayerMONOMIALRES}. This implies that, when the number of minimal generators $r$ is fixed, the maximum of the possible Betti numbers $\beta_{i+1}(M)$ for a monomial ideal $M \subset S$ for each homological degree $i+1$  is bounded by $c_i(n,r)$, the number of $i$-dimensional faces of the $n$-dimensional cyclic polytope with $r$ vertices. Let $\beta_{i+1}(n,r)$ be $\max_M \{ \beta_{i+1}(M) \} $ where the maximum is taken over all monomial ideals in $S$ with $r$ minimal generators. The remark we just made means that $\beta_{i+1}(n,r) \leq c_i(n,r)$ \cite[Theorem 6.3]{bayerMONOMIALRES}. In particular, for $n \geq 4$, $\beta_2(n,r) \leq { r \choose 2}$, and the extremal behavior of $\beta_2(n,r)$ has been characterized as a consequence of a result on the order dimension of the poset of the complete graph with $r$ vertices (see the discussion on page 134 of \cite{hosten+morris}). For instance, $\beta_2(4,r)$ attains this binomial upper bound for $4 \leq r \leq 12$, but $\beta_2(4,13) = 77 < 78 = {13 \choose 2}$. Similarly, $\beta_2(5,r) = c_1(5,r)$ for $ 5 \leq r \leq 81$, but $\beta_2(5,82) < c_1(5,82)$; and $\beta_2(6,r) = c_1(6,r)$ for $6 \leq r \leq 2646$, but  $\beta_2(6,2647) < c_1(6,2647)$. See \cite{OEIS} for more of this sequence. 

The plot in Figure \ref{fig:betti-plot} showcases the average behavior of $\beta_2(M)$, for $M$ generated by $r$ monomials in five indeterminates, compared to the upper bound $\beta_2(5,r) = {r \choose 2}$. We also include the experimental maximum second Betti number, taken over 1000 samples, for each $r$. Both the average and observed maximum $\beta_2$ grow approximately linearly and they are far from the real maximum for even moderate number of minimal generators. The extremal monomial ideals which give $\beta_2(n,r)$ seem to be truly extremal. We believe that similar computations will shed light on the behavior of $\beta_{i+1}(n,r)$.

The proof of Theorem \ref{thm:scarf} showed that for $p$ sufficiently large, $\beta_n(S/M)$ will be strictly greater than $f_{n-1}(\Delta_M)$. Figure \ref{fig:fvector-plot} suggests it may be possible to quantify this discrepancy. For example when $n=5$, $\expect{\beta_5}$ appears to grow linearly with the number of minimal generators, while $\expect{f_4}$ remains essentially constant. 
In fact, $\expect{\beta_i}$ looks remarkably well-behaved---even linear---for every $i$. These preliminary data suggest that average Betti numbers, as a function of $r$, may have strikingly different growth orders than their upper bounds.

\begin{figure}[H]
				\begin{center}
\includegraphics[width=.65\textwidth]{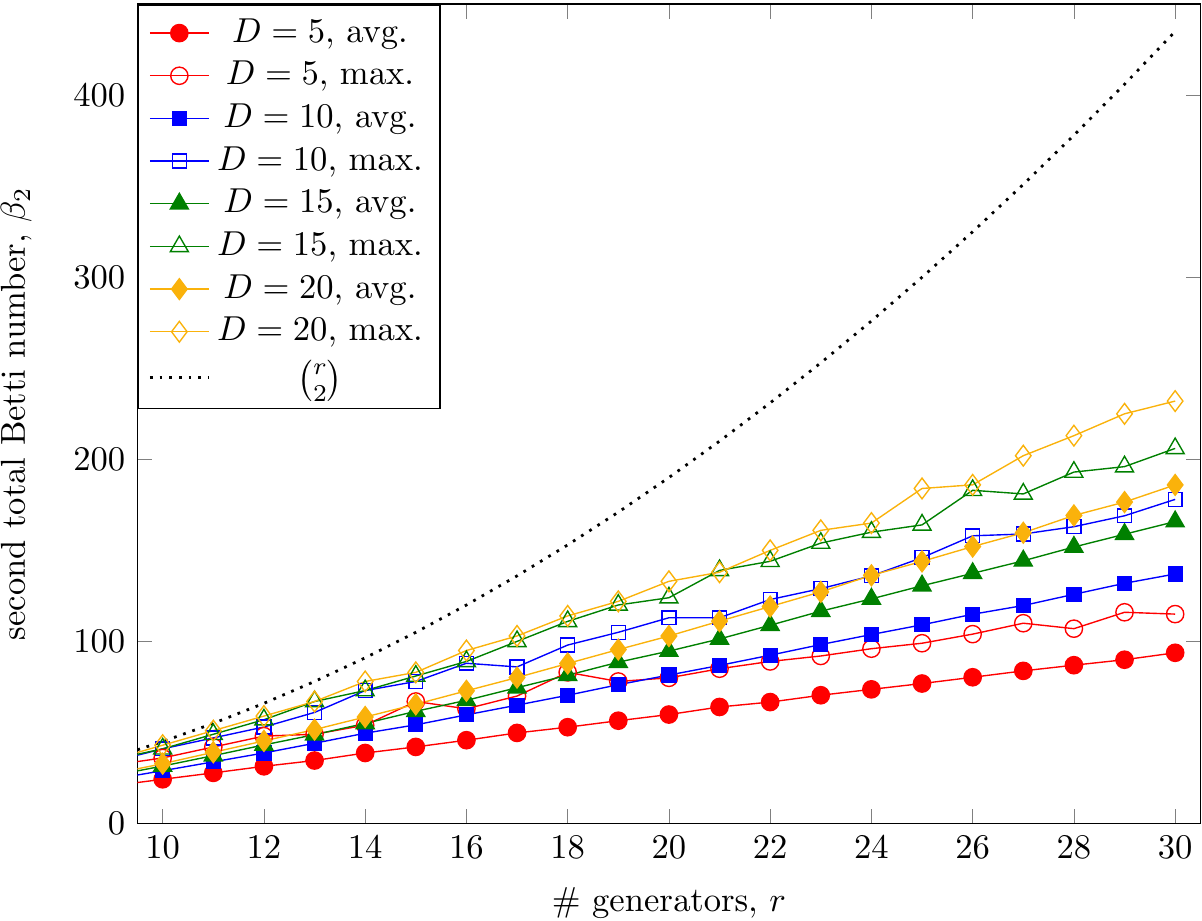}
	\caption{Average and maximum $\beta_2$ for $n=5$. Each value is based on 1000 random $M$.}
	\label{fig:betti-plot}
	\end{center}
\end{figure}
\begin{figure}[H]
	\begin{center}
		\includegraphics[width=.65\textwidth]{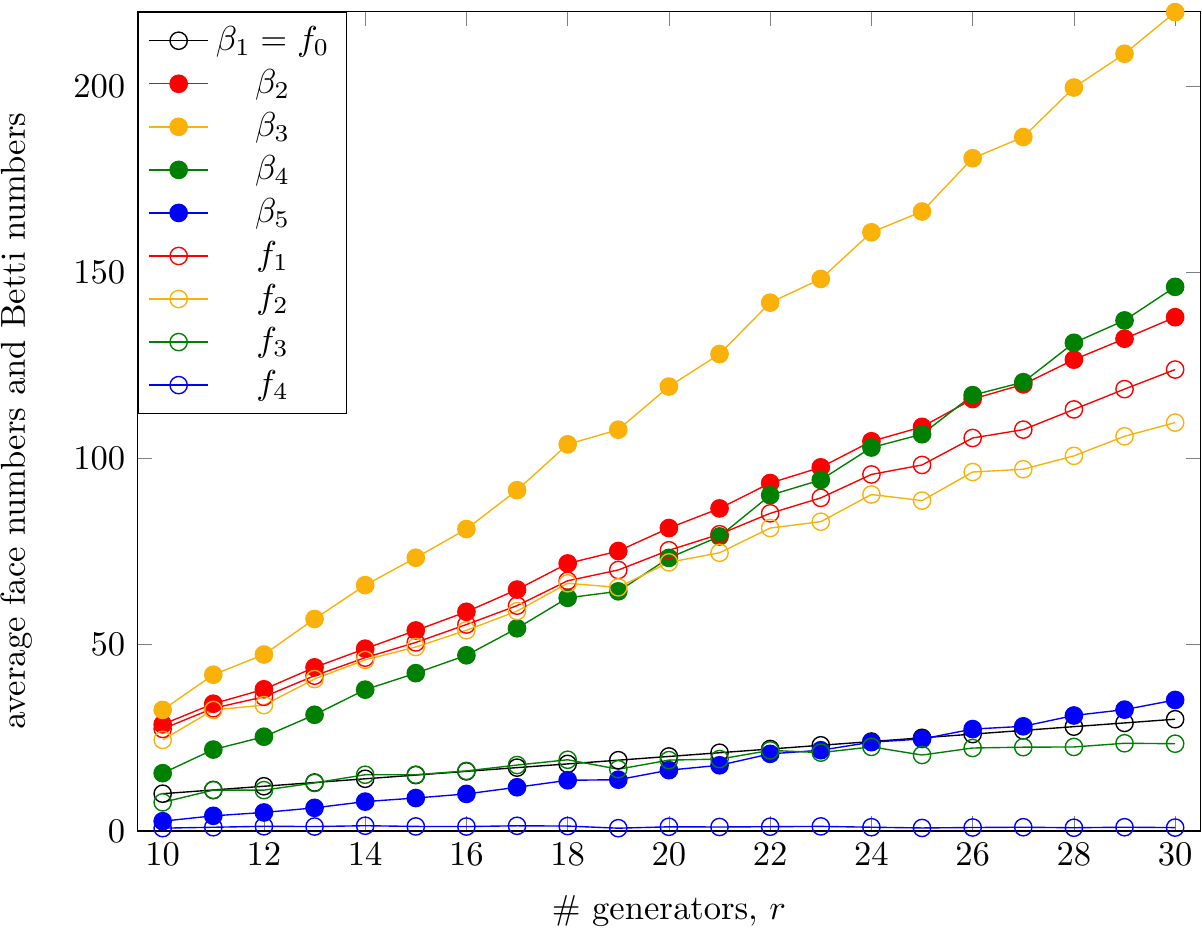}
	\caption{Average values of the Betti numbers, $\beta_1,\ldots,\beta_n$, and the Scarf complex face numbers, $f_0,\ldots,f_{n-1}$, for $n=5$ and $D=10$. Each average is based on 100 random $M$.}
\label{fig:fvector-plot}
	\end{center}
\end{figure}

\section{Acknowledgements} This work was conducted and prepared 
at the Mathematical Sciences Research Institute in Berkeley, California, during the Fall 2017 semester. 
Thus we gratefully acknowledge partial support by NSF grant DMS-1440140. In addition,
the first and fourth author were also partially supported by NSF grant DMS-1522158.

Thank you to our anonymous referees, whose thoughtful comments improved the final version of this paper.

Computer simulations made use of the {\tt Random Monomial Ideals}
package \cite{RMIpackage} for {\tt Macaulay2} \cite{M2}.

\bibliographystyle{amsplain}
\bibliography{pdim}

\providecommand{\bysame}{\leavevmode\hbox to3em{\hrulefill}\thinspace}
\providecommand{\MR}{\relax\ifhmode\unskip\space\fi MR }
\providecommand{\MRhref}[2]{%
  \href{http://www.ams.org/mathscinet-getitem?mr=#1}{#2}
}
\providecommand{\href}[2]{#2}
\begin{thebibliography}{10}

\bibitem{alesandroniDOM}
Guillermo Alesandroni, \emph{Minimal resolutions of dominant and semidominant
  ideals}, J. Pure Appl. Algebra (2017), 780--798.

\bibitem{alesandroniPDIM}
\bysame, \emph{Monomial ideals with large projective dimension}, arXiv preprint
  arXiv:1710.05124 (2017).

\bibitem{bayerMONOMIALRES}
Dave Bayer, Irena Peeva, and Bernd Sturmfels, \emph{Monomial resolutions},
  Math. Res. Lett. \textbf{5} (1998), no.~1--2, 31--46.

\bibitem{bigatti}
Anna~Maria Bigatti, \emph{Upper bounds for the {B}etti numbers of a given
  {H}ilbert function}, Comm. Algebra \textbf{21} (1993), no.~7, 2317--2334.
  \MR{1218500}

\bibitem{boocher2018lower}
Adam Boocher and James Seiner, \emph{Lower bounds for betti numbers of monomial
  ideals}, Journal of Algebra \textbf{508} (2018), 445--460.

\bibitem{brun2004betti}
Morten Brun and Tim R{\"o}mer, \emph{Betti numbers of $\mathbb{Z}^n$-graded
  modules}, Communications in Algebra \textbf{32} (2004), no.~12, 4589--4599.

\bibitem{buchsbaum+eisenbud}
David~A. Buchsbaum and David Eisenbud, \emph{Algebra structures for finite free
  resolutions, and some structure theorems for ideals of codimension {$3$}},
  Amer. J. Math. \textbf{99} (1977), no.~3, 447--485. \MR{0453723}

\bibitem{charalambous1991betti}
Hara Charalambous, \emph{Betti numbers of multigraded modules}, Journal of
  Algebra \textbf{137} (1991), no.~2, 491--500.

\bibitem{coxlittleoshea}
David Cox, John~B. Little, and Donal O'Shea, \emph{Ideals, varieties, and
  algorithms: An introduction to computational algebraic geometry and
  commutative algebra}, Springer, 2007.

\bibitem{deloeraRMI}
Jes\'{u}s~A. {De Loera}, Sonja {Petrovic}, Lily {Silverstein}, Despina {Stasi},
  and Dane {Wilburne}, \emph{Random monomial ideals}, to appear in Journal of
  Algebra, available at arXiv:1701.07130 (2017).

\bibitem{ein+erman+lazarsfeld}
Lawrence Ein, Daniel Erman, and Robert Lazarsfeld, \emph{Asymptotics of random
  {B}etti tables}, J. Reine Angew. Math. \textbf{702} (2015), 55--75.

\bibitem{Eisenbudbook}
David Eisenbud, \emph{Commutative algebra: with a view toward algebraic
  geometry}, Graduate Texts in Mathematics, vol. 150, Springer-Verlag, New
  York, 1995. \MR{1322960}

\bibitem{eisenbud+schreyer}
David Eisenbud and Frank-Olaf Schreyer, \emph{Betti numbers of graded modules
  and cohomology of vector bundles}, J. Amer. Math. Soc. \textbf{22} (2009),
  no.~3, 859--888. \MR{2505303}

\bibitem{ermanFLAG}
Daniel Erman and Jay Yang, \emph{Random flag complexes and asymptotic
  syzygies}, arXiv preprint arXiv:1706.01488 (2017).

\bibitem{M2}
Daniel~R. Grayson and Michael~E. Stillman, \emph{Macaulay2, a software system
  for research in algebraic geometry}, Available at
  https://faculty.math.illinois.edu/Macaulay2/.

\bibitem{hartshornePROBLEMS}
Robin Hartshorne, \emph{Algebraic vector bundles on projective spaces: a
  problem list}, Topology \textbf{18} (1979), no.~2, 117--128. \MR{544153}

\bibitem{herzog+hibi}
J\"urgen Herzog and Takayuki Hibi, \emph{Monomial ideals}, Graduate Texts in
  Mathematics, vol. 260, Springer-Verlag London, Ltd., London, 2011.
  \MR{2724673}

\bibitem{hosten+morris}
Serkan Ho\c{s}ten and Walter~D. Morris, Jr., \emph{The order dimension of the
  complete graph}, Discrete Math. \textbf{201} (1999), no.~1-3, 133--139.
  \MR{1687882}

\bibitem{hulett}
Heather~A. Hulett, \emph{Maximum {B}etti numbers of homogeneous ideals with a
  given {H}ilbert function}, Comm. Algebra \textbf{21} (1993), no.~7,
  2335--2350. \MR{1218501}

\bibitem{lascala+stillman}
Roberto La~Scala and Michael Stillman, \emph{Strategies for computing minimal
  free resolutions}, J. Symbolic Comput. \textbf{26} (1998), no.~4.

\bibitem{millerCCA}
Ezra Miller and Bernd Sturmfels, \emph{Combinatorial commutative algebra},
  Graduate Texts in Mathematics, vol. 227, Springer New York, 2004.

\bibitem{pardue}
Keith Pardue, \emph{Deformation classes of graded modules and maximal {B}etti
  numbers}, Illinois J. Math. \textbf{40} (1996), no.~4, 564--585. \MR{1415019}

\bibitem{RMIpackage}
Sonja {Petrovi{\'c}}, Despina {Stasi}, and Dane {Wilburne}, \emph{{Random
  Monomial Ideals Macaulay2 Package}}, ArXiv e-prints (2017).

\bibitem{OEIS}
Neil~J.A. Sloane, \emph{The {O}nline {E}ncyclopedia of {I}nteger {S}equences,
  {A}001206}, Available at https://oeis.org/A001206.

\bibitem{walkerBEH}
Mark~E. Walker, \emph{Total {B}etti numbers of modules of finite projective
  dimension}, Ann. of Math. (2) \textbf{186} (2017), no.~2, 641--646.
  \MR{3702675}

\end{thebibliography}
\end{document}